\documentclass[11pt]{amsart}
\usepackage{url}
\usepackage{hyperref,amssymb}
\hypersetup{colorlinks=true,linkcolor=blue,
urlcolor=cyan,citecolor=magenta} 
\usepackage[hyperpageref]{backref}
\usepackage{frcursive,calrsfs,mathrsfs}
\newtheorem{theorem}{Theorem}[section]
\newtheorem{lemma}[theorem]{Lemma}
\newtheorem{corollary}[theorem]{Corollary}

\newtheorem{remark}[theorem]{Remark}
\newtheorem{remarks}[theorem]{Remarks}

\numberwithin{equation}{section}
\newcommand{\Q}{{\mathbb{Q}}}
\newcommand{\Z}{{\mathbb{Z}}}
\newcommand{\F}{{\mathbb{F}}}

\newcommand{\ds}{\displaystyle}
\newcommand{\ov}{\overline}
\newcommand{\wt}{\widetilde}
\newcommand{\ft}{\footnotesize}
\newcommand{\ns}{\normalsize}
\newcommand{\CE}{{\mathcal E}}
\newcommand{\CG}{{\mathcal G}}
\newcommand{\CH}{{\mathcal H}}
\newcommand{\CR}{{\mathcal R}}
\newcommand{\CS}{{\mathcal S}}
\newcommand{\CT}{{\mathcal T}}
\newcommand{\CU}{{\mathcal U}}
\newcommand{\CW}{{\mathcal W}}
\newcommand{\CX}{{\mathcal X}}
\newcommand{\order}{\raise0.8pt \hbox{${\scriptstyle \#}$}}
\newcommand{\lien}{\mathrel{\mkern-4mu}}
\newcommand{\too}{\relbar\lien\rightarrow}

\newcommand{\plus}{\ds\mathop{\raise 0.5pt \hbox{$\bigoplus$}}\limits}
\newcommand{\prd}{\ds\mathop{\raise 1.0pt \hbox{$\prod$}}\limits}
\newcommand{\sm}{\ds\mathop{\raise 1.0pt \hbox{$\sum$}}\limits}
\newcommand{\ffrac}[2]{\hbox{\ft $\displaystyle\frac{#1}{#2}$}}
\newcommand{\Gal}{{\rm Gal}}
\newcommand{\ram}{{\rm ram}}
\newcommand{\ab}{{\rm ab}}

\newcommand{\nr}{{\rm nr}}
\newcommand{\gen}{{\rm gen}}
\newcommand{\pr}{{\rm pr}}
\newcommand{\bp}{{\rm bp}}
\newcommand{\ta}{{\rm ta}}
\newcommand{\rk}{{\rm rk}}
\newcommand{\1}{{\chi_0^{}}}
\newcommand{\Cha}{{\rm III}}
\newcommand{\Hom}{{\rm H}}

\newcommand{\Norm}{\hbox{\bf$\textbf{N}$}}
\newcommand{\BC}{\hbox{\bf$\textbf{C}$}}
\newcommand{\BD}{\hbox{\bf$\textbf{D}$}}
\newcommand{\BH}{\hbox{\bf$\textbf{H}$}}
\newcommand{\BJ}{\hbox{\bf$\textbf{J}$}}
\newcommand{\BK}{\hbox{\bf$\textbf{K}$}}
\newcommand{\BL}{\hbox{\bf$\textbf{L}$}}
\newcommand{\BR}{\hbox{\bf$\textbf{R}$}}
\newcommand{\BT}{\hbox{\bf$\textbf{T}$}}
\newcommand{\BV}{\hbox{\bf$\textbf{V}$}}
\newcommand{\BX}{\hbox{\bf$\textbf{X}$}}
\newcommand{\BZ}{\hbox{\bf$\textbf{Z}$}}

\author[Georges Gras]{Georges Gras}

\address{\hspace{-0.5cm}Villa la Gardette, 4 chemin Ch\^ateau Gagni\`ere, 
38520 Le Bourg d'Oisans \rm {\url{http://orcid.org/0000-0002-1318-4414} } }
\email{g.mn.gras@wanadoo.fr}

\keywords{abelian fields; $\BK$-theory; Hilbert's kernel; $p$-adic $\zeta$ and 
$L$-functions, Stickelberger's element, class field theory, PARI programs}

\subjclass{11R42, 11R70, 11S40, 19F15, 11R37}

\begin{document}

\title[Genus theory of $p$-adic pseudo-measures]
{Genus theory of $p$-adic pseudo-measures \\ Hilbert's kernels \& abelian $p$-ramification}

\date{January 17, 2024}

\begin{abstract} We consider, for real abelian fields $K$, the Birch--Tate 
formula linking $\order \BK_2(\BZ_K)$ to $\zeta_K(-1)$; 
we compare, for quadratic and cyclic cubic fields with $p \in \{2, 3\}$, 
$\order \BK_2(\BZ_K)[p^\infty]$ to the order of the torsion group 
$\CT_{K, p}$ of abelian $p$-ramification theory given by the residue 
of $\zeta_{K, p}(s)$ at $s=1$. This is done via the ``genus theory'' of 
$p$-adic pseudo-measures, inaugurated in the 1970/80's and the fact 
that $\CT_{K, p}$ only depends on the $p$-class group and on the 
normalized $p$-adic regulator of $K$ (Theorem \ref{main}\,{\bf A}). We 
apply this to prove a conjecture of Deng--Li giving the structure of 
$\BK_2(\BZ_K)[2^\infty]$ for an interesting family of real quadratic 
fields (Theorem \ref{main}\,{\bf B}). Then, for $p \geq 5$, we give a 
lower bound of $\rk_p(\BK_2(\BZ_K))$ in cyclic $p$-extensions $K/\Q$
(Theorem \ref{main}\,{\bf C}). Complements, PARI programs and tables 
are given in Appendices.
\end{abstract}

\maketitle 

\tableofcontents

\section{Introduction and prerequisites}
The aim of this paper is to utilize the link between $p$-adic $\BL$-functions
$\BL_p(s, \chi)$, of even Dirichlet's characters $\chi = \varphi \psi$ ($\varphi$ 
of prime-to-$p$ order, $\psi$ of {\it non-trivial} $p$-power order), and classical 
arithmetic $p$-invariants of the corresponding cyclic number field $K$, then to 
show that these invariants have some connections due to the character $\psi$ 
(genus theory principle). This connection does exist because of reflection 
theorems of class field theory and this allows to get additional informations.

\smallskip
Let's give some notations for the description of these methods and the
statement of the results:

\subsection{Notations}\label{notations}
(i) Let $\chi =: \varphi \psi$ be an even character, where $\varphi$ is of 
prime-to-$p$ order and $\psi$ of $p$-power order $p^e$, $e \geq 1$. 
For $a \in \Z_p^\times$, let $\theta (a)$ be the unique $\xi \in {\rm tor}_{\Z_p}(\Z_p^\times)$ 
such that $a \equiv \xi \pmod p$ for $p \ne 2$ (resp. $a \equiv \xi \pmod 4$ for $p=2$).
The character $\theta$ will be considered as the cyclotomic character of $\Q(\mu_q)$ 
($q = p$ or $4$).

\smallskip
Let $K$ be the real cyclic field fixed by $\chi$ and let $M$ be the 
subfield of $\Q^\ab$ fixed by the character $\theta^{-1}\,\chi$;
since $K$ is real, $M$ is an imaginary field.
Let $M_0$ be the maximal subfield of $M$ of prime-to-$p$ degree:

\smallskip
\quad $\bullet$ for $p \ne 2$, then $M \subseteq K(\mu_p)$, and $M_0$ is the  fixed field of
$\theta^{-1}\,\varphi$ since $\theta$ is of order $p-1$; if $\varphi = \1$ (the unit character), 
$M_0=\Q(\mu_p)$;

\smallskip
\quad$\bullet$ for $p = 2$, then $M \subset K(\sqrt {-1})$, and $M_0$ is the fixed field 
of $\varphi$ since $\theta$ is of order $2$; if $\varphi = \1$, $M_0=\Q$.

\smallskip
(ii) Let ${\mathfrak m}$ be the maximal ideal of $\Q_p(\mu_{p^e})$ and let
$v_{\mathfrak m}$ be the corresponding valuation with image $\Z$ (so, 
$v_{\mathfrak m} = (p-1)\,p^{e-1} \cdot v_p$).

\smallskip
(iii) For any prime number $r$, let $p^{d_r}$ be the degree of the splitting field
of $r$ in $M/M_0$; let $\BD = \sum_{\ell} p^{d_\ell}$, for $\ell \ne p$
ramified in $M/M_0$ and totally split in $M_0/\Q$; whence for $\ell \ne p$
such that:

\smallskip
\quad $\bullet$ for $p \ne 2$: $\psi(\ell) = 0$ and $(\theta^{-1}\,\varphi) (\ell) = 1$;

\smallskip
\quad $\bullet$ for $p = 2$: $\psi(\ell) = 0$ and $\varphi(\ell) = 1$ (let $\widehat K$ 
with $[K : \widehat K] = 2$; then $M$ is the quadratic extension of $\widehat K$
distinct from $K$ and $\widehat K(\sqrt {-1})$).

\subsection{Overview of the method}\label{method}
Genus theory does exist for $\BL_p$-functions, 
as this was initiated in the 1970/80's with the genus theory of $p$-adic abelian 
pseudo-measures \cite[Th\'eor\`eme (0.3)]{Gr33}, \cite[Th\'eor\`eme (0.1)]{Gr34} 
including the tricky split case $\theta^{-1}\varphi (p) = 1$.
The results of genus theory under consideration claim that:
\begin{equation}\label{either/or}
\left\{\begin{aligned}
\hbox{either}\ \ \ \ v_{\mathfrak m} \big(\hbox{$\frac{1}{2}$}\,\BL_p(s, \chi) \big) 
& > \BC(s),\,\forall s \in \Z_p, \\
\hbox{or}\ \ \ \ v_{\mathfrak m} \big(\hbox{$\frac{1}{2}$}\,\BL_p(s, \chi) \big) 
& = \BC(s),\,\forall s \in \Z_p, 
\end{aligned} \right.
\end{equation}
for an explicit $\BC(s)$, which is a constant $\BC$ as soon as 
$\theta^{-1}\varphi (p) \ne1$ (see Theorem \ref{thmfond} defining
$\BC(s)$, essentially depending on $\BD$). 

\smallskip
This allows to assign, for $p \in \{2, 3\}$, similar properties to the following 
arithmetic $p$-invariants depending on values of $\BL_p$-functions:

\smallskip
({\bf a}) We will consider, at first (Section \ref{torsion}), the torsion group $\CT_{K, p}$ 
of the Galois group of the maximal abelian $p$-ramified (non-complexi\-fied)
pro-$p$-extension $H_{K, p}^\pr$ of $K$. 
In the real abelian case, $\CT_{K, p} = \Gal(H_{K, p}^\pr/K_\infty)$, where $K_\infty$ 
is the cyclotomic $\Z_p$-exten\-sion of $K$. 

\smallskip
Analytically (where $\sim$ means equality ``up to a $p$-adic unit factor''): 
$$\order \CT_{K, p} \sim [K \cap\,\Q_\infty : \Q] \times \prd_{\chi \ne 1} 
\hbox{$\frac{1}{2}$}\,\BL_p (1, \chi), $$ 
where $\chi$ runs trough the set of primitive Dirichlet's characters of $K$ 
\cite[Appendix]{C}, \cite{Se}, while class field theory elucidates it completely 
with Formula \eqref{T=HRW} \cite{Gra8}, which explains that $\CT_{K, p}$ 
plays a central role in this study. 

\smallskip
({\bf b}) Then we will focus (Section \ref{hilbert}) on the the Hilbert tame kernel $\BK_2(\BZ_K)$ 
of the ring of integers $\BZ_K$ of $K$ and, for $p=2$, we will replace it by the regular 
kernel $\BR_2(\BZ_K)$ in the ordinary sense, linked to the Hilbert one via 
the exact sequence \cite{Gar}:
\begin{equation}\label{regular}
1 \to \BR_2(\BZ_K) \too \BK_2(\BZ_K) \too (\Z/2\Z)^{[K\,:\,\Q]} \to 1.
\end{equation}

The class field theory information about the regular kernel comes in general
from reflection theorems giving formulas \eqref{rankR} and \eqref{rankT}; we have 
refrained from citing all the articles re-proving well-known reflection theorems between
$\BK_2(\BZ_K)[p^\infty]$, $\CT_{K,p}$, $\CH_{K,p}$, published in the 1980--90's 
by several authors.

\smallskip
({\bf c}) Then $\order \BK_2(\BZ_K)$, being in relation with the complex $\BL(-1, \chi)$'s, 
we use (Section \ref{comparison}), when it is possible, the classical link between 
$p$-adic and complex $\BL$-functions for some comparison of the two $p$-adic 
invariants $\BK_2(\BZ_K)[p^\infty]$ and $\CT_{K, p}$\,\footnote{\,A $p$-invariant 
may be denoted $\BX_K[p^\infty]$ when it is of the form $\BX_K \otimes \Z_p$, for a 
{\it finite global invariant $\BX_K$}. Otherwise it is denoted $\CX_{K, p}$;
it is only conjectural \cite[Section 8]{Gr55} that $\CT_{K, p}$ can be written 
$\BT_K \otimes \Z_p$; this conjecture is out of reach but may be coherent 
with the finiteness of $\BH_K$ (class group) and $\BK_2(\BZ_K)$, 
and is supported by many heuristics. For convenience, we 
will also use the notation $\CX_{K, p}$ when $\BX_K$ does exist (e.g., 
$\CH_{K, p} = \BH_K[p^\infty]$).}, 
with the use of property \eqref{either/or} when it applies, thus giving new informations.

\subsection{Torsion group of abelian \texorpdfstring{$p$}{Lg}-ramification}
\label{torsion}
Class field theory is especially convenient to state the arithmetic properties of 
$\CT_{K, p}$ (e.g., \cite[\S\,III.2.c, \S\,III.4.b, \S\,IV.3]{Gr5}, after some pioneering 
works \cite{BP, Gr3, Jau1, Ng1, GJ} and many others). 

\smallskip
Numerical aspects are given in \cite{Gr6} with PARI \cite{PARI} programs and 
the order of magnitude of $\order \CT_{K, p}$ is studied in \cite{Gra9}. 

\smallskip
We refer to Appendix \ref{A1} for the program computing the structure of $\CT_{K, p}$ 
in complete generality.
The following diagram is valid for any number field $K$ fulfilling Leopoldt's conjecture, 
replacing $K_\infty$ by the compositum $\wt K$ of the $\Z_p$-extensions of $K$:
\unitlength=0.98cm 
\[\vbox{\hbox{\hspace{-2.8cm} 
 \begin{picture}(11.5, 5.4)
\put(6.1, 4.50){\line(1, 0){1.6}}
\put(8.4, 4.50){\line(1, 0){2.3}}
\put(3.8, 4.50){\line(1, 0){1.2}}
\put(9.1, 4.2){\ft$\simeq {\mathcal W}_{K, p}$}
\put(4.2, 2.50){\line(1, 0){1.25}}
\bezier{350}(3.8, 4.8)(7.6, 5.5)(11.0, 4.8)
\put(7.2, 5.25){\ft$\CT_{K, p}$}
\bezier{350}(3.7, 4.2)(5.85, 3.4)(8.0, 4.2)
\put(6.0, 3.5){\ft$\CT_{K, p}^\bp$}
\put(3.50, 2.9){\line(0, 1){1.25}}
\put(3.50, 0.9){\line(0, 1){1.25}}
\put(5.7, 2.9){\line(0, 1){1.25}}
\bezier{300}(3.9, 0.5)(4.7, 0.5)(5.6, 2.3)
\put(5.2, 1.3){\ft$\simeq \CH_{K, p}$} 
\bezier{300}(6.0, 2.5)(8.5, 2.6)(10.8, 4.3)
\put(8.2, 2.7){\ft$\simeq  \CU_{K, p}/\CE_{K, p}$}
\put(10.85, 4.4){\ft$H_{K, p}^\pr$}
\put(5.2, 4.4){\ft$\wt K H_{K, p}^\nr$}
\put(7.8, 4.4){\ft$H_{K, p}^\bp$}
\put(6.4, 4.2){\ft$\simeq {\mathcal R}_{K, p}$}
\put(3.4, 4.4){\ft$\wt K$}
\put(5.5, 2.4){\ft$H_{K, p}^\nr$}
\put(2.8, 2.4){\ft$H_{K, p}^\nr\! \cap\! \wt K$}
\put(3.4, 0.40){\ft$K$}
\end{picture}}}\]

Definitions and notations are the following:

\smallskip
(i) $\CH_{K, p}$ is the $p$-class group and $H_{K, p}^\nr$ the $p$-Hilbert class field;

\smallskip
(ii) $\CR_{K, p}$ is the {\it normalized} $p$-adic regulator, defined as a $\Z_p$-module 
in \cite[Section 5, Proposition 5.2]{Gra8}; its order is the classical $p$-adic regulator, 
up to an explicit factor;

\smallskip
(iii) $\CU_{K, p}$ is the group of principal local units at $p$, $\CE_{K, p}$ the closure in $\CU_{K, p}$ 
of the group of units $E_K$ and $\CW_{K, p} := {\rm tor}_{\Z_p}(\CU_{K, p})/\mu_p(K)$;

\smallskip
(iv) $H_{K, p}^\bp$ (resp. $\CT_{K, p}^\bp$), named in \cite{Ng1} the Bertrandias--Payan 
field (resp. module), is issued from \cite{BP}; $H_{K, p}^\bp$ is the compositum of the 
$p$-cyclic extensions of $K$ embeddable in $p$-cyclic extensions of arbitrary large 
degree; see \cite{GJN} for complements. Whence:
\begin{equation}\label{T=HRW}
\order \CT_{K, p} = \ffrac{\order \CH_{K, p}}{[H_{K, p}^\nr \cap \wt K : K]} 
\times \order \CR_{K, p} \times \order \CW_{K, p}.
\end{equation}

In the real case, $K_\infty/K$ is in general totally ramified, 
so that the formula becomes $\order \CT_{K, p} = \order \CH_{K, p} \times
\order \CR_{K, p} \times \order \CW_{K, p}$.

\begin{remark}
From \cite[Proposition 5.2]{Gra8}, we can give the following examples:

\smallskip
(i) For $p=2$ and $K = \Q(\sqrt m)$ ($m>0$ square-free), this formula 
becomes $\order \CT_{K, 2} \sim h_K \times \ffrac{\log_2(\varepsilon_K)}{2 \sqrt m}$, 
where $h_K$ is the class number and $\varepsilon_K$ the fundamental unit.

\smallskip
(ii) For $p=3$ and a cyclic cubic field $K$, then $\order \CT_{K, 3} \sim h_K \times 
\ffrac{R_{K, 3}}{3^u}$ with $u=2$ (resp. $u=1$) if $3$ is unramified (resp. ramified) 
and where $R_{K, 3}$ is the usual $3$-adic regulator
\end{remark}

\subsection{Tame and regular Hilbert's kernels}\label{hilbert}

Let $K$ be any number field, let $p \geq 2$ be a prime number and let $S_p(K)$ be the set of 
$p$-places of $K$; put $K' := K(\mu_p)$ and let $S_p(K')$ be the set of $p$-places of $K'$. 

\smallskip
Recall that $\omega : \Gal(\Q(\mu_p)/\Q) \to \mu_{p-1}$, called the cyclotomic (or Teichm\"uller) 
character, is defined by the Galois action on $\mu_p$ (conductor $p$, order $p-1$ for 
$p \ne 2$); it is the unit character $\1$ for $p=2$ since $\mu_2 \subset \Q$. For $p \ne 2$, 
we have in some sense $\omega = \theta$ (Notations \ref{notations}\,(i)), but for $p=2$, 
do not confuse $\omega$ with $\theta$ as cyclotomic character of $\Gal(\Q(\mu_4)/\Q)$ 
(conductor $4$, order$2$).

\subsubsection{First reflection theorem}
The regular kernel $\BR_2(\BZ_K)$ (see \eqref{regular}) fulfills, from Tate's results \cite{Ta1, Ta2} 
and reflection theorem \cite[Th\'eor\`eme 11.1]{Gr4} or \cite[Theorem II.7.7.3.1]{Gr5}, 
the following $p$-rank formula, where $\BH_{K'}^{S_p(K') {\rm res}}$ is the 
$S_p(K')$-class group in the restricted sense, $\delta =1$ or$0$, according as 
$\mu_p \subset K$ or not:
\begin{equation}\label{rankR}
\left\{ \begin{aligned}
&\rk_p(\BR_2(\BZ_K))  =  \rk_p \big((\BH_{K'}^{S_p(K') {\rm res}})_{\omega^{-1}} \big) \\
& \hspace{1cm}+\order \{v \in S_p(K),\ \hbox{$v$ totally split in $K'/K$} \}-\delta, \\
& \rk_2(\BR_2(\BZ_K)) =  \rk_2 \big(\BH_{K}^{S_2 {\rm res}} \big) 
+\order S_2-1.
\end{aligned}\right.
\end{equation}

\subsubsection{Second reflection theorem}
Due to the reflection theorem between $p$-class groups and $p$-torsion groups
of abelian $p$-ramification theory, we get similarly \cite[Proposition III.4.2.2]{Gr5}:
\begin{equation}\label{rankT}
\left\{\begin{aligned}
& \rk_p(\CT_{K, p}) = \rk_p \big((\BH_{K'}^{S_p(K') {\rm res}})_\omega \big) \\
&  \hspace{1cm}+\order \{v \in S_p(K),\,\hbox{$v$ totally split in $K'/K$} \}-\delta \\
& \rk_2(\CT_{K, 2}) =  \rk_2 \big(\BH_{K}^{S_2 {\rm res}} \big) 
+\order S_2-1.
\end{aligned}\right.
\end{equation}

From the two previous relations, we get:
\begin{equation}\label{equalranks}
\rk_p (\BR_2(\BZ_K)) = \rk_p (\CT_{K, p}), 
\end{equation}
as soon as $\omega^2=1$, that is to say, if $K$ contains the maximal real 
subfield of $\Q(\mu_p)$, a framework widely developed in \cite{GJ};
these techniques yielding generalizations as that of Keune \cite{K}.

\smallskip
The relation was proven in \cite[Theorems 1, 2]{Gr3} (in which $\BR_2(\BZ_K)$ 
was denoted $H_2^0K$), to characterize, for $p \in \{2, 3\}$, the abelian $p$-extensions 
$K/\Q$ such that $\BR_2(\BZ_K)[p^\infty] = 1$ ($p$-regular fields studied in
\cite{GJ, BGr, RO}).

\smallskip
We note that $\Q(\mu_p)^+ \subseteq K$ is always fulfilled for $p \in \{2, 3\}$.
More generally, \cite[Remarque 11.5]{Gr4} gives, when $K$ contains $\mu_p$, 
the corresponding relation \eqref {equalranks} with characters: 
$$\rk_\chi (\BR_2(\BZ_K)) = \rk_{\omega^2 \chi^{-1}} (\CT_{K, p}). $$

Class field theory approach of the structures of $\BK_2(\BZ_K)[p^\infty]$ and
$\CT_{K, p}$ has given a huge literature, especially for rank 
computations, often restricted to $p \in \{2, 3\}$, probably with  
\eqref{equalranks} when $\omega^2 = \1$, to get practical results (e.g., 
\cite{BS, Gr3, K, Ber, Br1, Q1, YF, KM, Br2, Yue2, Q2, DL}).

\smallskip
For results in the case $p \geq 5$, see Section  \ref{p>3}.

\subsection{Arithmetic v.s. analytic properties}\label{comparison}

To complete these class field theory aspects, we will use the following 
($p$-adic and complex) analytic formulas and the correspondence, for 
$p \in \{2, 3\}$, between $p$-adic and complex $\BL$-functions
of characters $\chi$ of $K$:
\begin{equation}\label{analytic}
\left\{\begin{aligned}
\order \CT_{K, p} & \sim [K \cap\,\Q_\infty : \Q] \times \prd_{\chi \ne 1} 
\ffrac{1}{2}\,\BL_p (1, \chi), \\
\order \BK_2(\BZ_K) & = 2^{[K : \Q]}\times \Big( \ffrac{1}{6}  \prd_{[F:K]=2} 
\ffrac{\order \mu(F)}{2}   \Big) \times \prd_{\chi \ne 1}\ffrac{1}{2}\,\BL (-1, \chi), 
\end{aligned}\right.
\end{equation}

\noindent
where $\mu(F)$ is the group of roots of unity of the field $F$. 

\smallskip
For more history and theoretical contributions, one may refer to 
\cite{Ba, Ta1, Gar, Ta2, Jau1, Ng1, Hur, GJ, K, Ko2, Ng, 
HK, RW, Ko1, Yue1, Gr5, Q3}, among many others, dealing with the more 
concrete results we have cited previously; for some generalizations to the 
wide \'etale kernels ${\bf WK}_{2i}(K)$, see for instance \cite{Ng2, JS, JM, AAM}.

\smallskip
But our goal is to return to very simple effective $p$-adic methods allowing 
computations, by means of pseudo-measures. We intend to provide, simultaneously,
survey and history parts about these various questions.

\smallskip
These pseudo-measures are the Mellin transform (via a reflection principle) of 
Stickelberger elements, and give rise, after twisting, to $p$-adic measures of the 
form $\big(\CS^*_{\!L_n}(c)(s) \big)_n$, $s \in \Z_p$, in the group algebras 
$\Z_p/qp^n \Z_p[\Gal(L_n/\Q)]$ of the layers $L_n := K(\mu_{qp^n})$ of the cyclotomic 
$\Z_p$-extension of $K(\mu_q)$ ($q=4$ or $p$) \cite[Section II]{Gr1}.

\smallskip
This was done after the main pioneering works of Kubota--Leopoldt \cite{KL}, 
then Fresnel \cite{Fre}, Amice--Fresnel \cite{AF}, Iwasawa \cite{Iw}, Coates \cite{C}, 
Serre \cite{Se}, giving many results, as the value of the residue of Dedekind
$p$-adic $\zeta$-functions at $s=1$ and some annihilation theorems for $p$-class 
groups, torsion groups $\CT_{K, p}$ and similar invariants \cite{Gr1, Gr2, Jau2, BN, 
Jau3, Jau4}; this was considerably generalized to a totally real base field by
Colmez in \cite{Col} and Deligne--Ribet in \cite{DR}. 

\smallskip
The Iwasawa framework of these constructions is very similar and
is detailed in \cite[\S\,7.2]{W}.

\subsection{Main results of the article}\label{main}
We describe three results about the study and the comparison of the
modules $\BK_2(\BZ_K)[p^\infty]$ and $\CT_{K, p}$:

\medskip
{\bf (a)} From the properties of the constant $\BC$ described in \S\,\ref{method}
and stated in Theorem \ref{thmfond}, we obtain the following results for 
quadratic fields ($p=2$) and cyclic cubic ones ($p=3$):

\subsubsection*{\bf Theorem A} \label{A} (see Theorem \ref{thmfinal} for the proof).

\smallskip
(i) {\it Let $K = \Q(\sqrt m)$ be a real quadratic field of conductor $\ne 8$.}

\smallskip
{\it Assume that $v_2(\order \CT_{K, 2}) = \BC$
(equivalent to $v_2 \big(\frac{1}{2} \BL_2(1, \chi) \big) = \BC$ 
from formula \eqref{analytic}); then}
$\order \BK_2(\BZ_K)[2^\infty] = 2^{\hbox{\tiny $\BC$}+2}$.

\smallskip
(ii)  {\it Let $K$ be a cyclic cubic field of conductor $\ne 9$.}

 {\it Assume that $v_3(\order \CT_{K, 3}) = \BC$
(so, $v_3 \big(\BL_3(1, \chi)  \times\BL_3(1, \chi^2) \big) = \BC$);
then} $\order \BK_2(\BZ_K)[3^\infty] = 3^{\hbox{\tiny $\BC$}}$.

\smallskip
(iii)  {\it If $v_2 (\CT_{K, 2} ) > \BC$  (resp. $v_3 (\CT_{K, 3} ) > \BC$), }
$\order \BK_2(\BZ_K)[2^\infty] > 2^{\hbox{\tiny $\BC$}+2}$  {\it (resp.}
$\order \BK_2(\BZ_K)[3^\infty] > 3^{\hbox{\tiny $\BC$}}$).

\medskip
We compute, in Appendix \ref{B1}, $\order \BR_2(\BZ_K) [2^\infty]$ for quadratic
fields by means of explicit pseudo-measures yielding expressions of 
Theorem \ref{Lp} and Corollary \ref{s=-1}; this allows safe verifications.

\medskip
{\bf (b)} We apply the case of equality to a family of real quadratic fields introduced 
in \cite[Theorem 1.2]{DL} and prove the Deng--Li conjecture on 
$\BK_2(\BZ_K)[2^\infty]$ in the following Theorem\,{\bf B} (Section \ref{DengLi}):

\subsubsection*{\bf Theorem B}\label{B} {\it
Let $K = \Q(\sqrt m)$, $m= \ell_1 \ell_2 \cdots \ell_n$, $n$ even, with:

\smallskip
$\bullet$ $\ell_1 \equiv 3 \pmod 8$,\ \ $\ell_i \equiv 5 \pmod 8$ for $i \geq 2$, 

\smallskip
$\bullet$  $\big( \frac{\ell_1}{\ell_2} \big) = -1$,  $\big( \frac{\ell_1}{\ell_j} \big) = 1$ for $j \geq 3$, 
\ \ $\big( \frac{\ell_i}{\ell_j} \big) = -1$ for $2 \leq i < j \leq n$.

\smallskip
Then $\CH_{K, 2} \simeq (\Z/2\Z)^{n-1}$, $\CT_{K, 2} \simeq (\Z/2\Z)^{n-2} \times \Z/4\Z$, and:}
$$\BK_2(\BZ_K)[2^\infty] \simeq (\Z/2\Z)^{n-1} \times \Z/2^3 \Z\ \hbox{(Deng--Li conjecture).}$$

\medskip
{\bf (c)} The cases $p \in \{2, 3\}$ being very specific (thus well studied in the literature) 
because of the relation \eqref{flp} for $m=2$ between $p$-adic and complex 
$\BL$-functions, we have obtained the following Theorem\,{\bf C} when $p$ is 
distinct from $2$ and $3$ (see Theorem \ref{pdiv} for a much more general 
statement over a real base field $k$ of prime-to-$p$ degree):

\subsubsection*{\bf Theorem C} \label{C}
{\it Let $K/\Q$ be a cyclic $p$-extension for $p \geq 5$ and let $S_\ta$ be the 
set of primes $\ell \ne p$ ramified in $K/\Q$. 

\smallskip\noindent
Then $\rk_p(\BK_2(\BZ_K)) \geq 
\order S_\ta$. In particular, $\BK_2(\BZ_K)[p^\infty] = 1$ if and only if
$K$ is contained in the cyclotomic $\Z_p$-extension $\Q_\infty$ of $\Q$.}

\section{Computation of \texorpdfstring{$\order \BK_2 (\BZ_K)[p^\infty]$, 
$p \in \{2, 3\}$}{Lg}}

\subsection{Birch-Tate formula and \texorpdfstring{$\BL_p$}{Lg}-functions}
Let $\zeta_K(s)$ be the De\-dekind zeta-function of the real abelian number field $K$. 

\smallskip
Then $\zeta_K(s) = \prd_{\chi} \BL(s, \chi)$, as product of the complex $\BL$-functions, 
where $\chi$ runs trough the set of primitive Dirichlet's characters of $K$, for 
which $\BL(s, \1) = \zeta_\Q^{}(s)$ for the unit character.
In the $p$-adic context, definition of $\BL_p$-functions is as follows
(e.g., \cite[Section 5\,(a), Remarque]{Fre}, Amice--Fresnel \cite[Section 1]{AF}):
\begin{equation}\label{flp}
\left\{\begin{aligned}
& \BL_p(1-m, \chi) = (1- p^{m-1}\,\chi(p))\,\BL(1-m, \chi), 
\hbox{\ for:} \\
&\hspace{0.6cm} m > 1\  \&\ \,m \equiv 0 \pmod {(p-1)}\ \hbox{if $p > 2$}, \\
&\hspace{0.6cm} m > 1\  \&\ \,m \equiv 0 \pmod 2\ \hbox{if $p = 2$}. 
\end{aligned}\right .
\end{equation}

This allows to compute for instance $\BL(-1, \chi)$ in terms of  
$\BL_p$-func\-tions for $p=2$ and $p = 3$ since $m=2$ fulfills the congruent
conditions required in \eqref{flp}.

\smallskip
The Birch-Tate formula is the following equality (proved as consequence
of the Mazur--Wiles ``Main Theorem'' in abelian theory and complements
in the case $p=2$ \cite{Ko2}):
\begin{equation}\label{BT}
\left\{\begin{aligned}
\order \BK_2 (\BZ_K) & = w_2(K)\,\zeta_K(-1) = w_2(K)\,\zeta_\Q(-1)\!\!
\prd_{\chi \ne \1} \BL(-1, \chi), \\
w_2(K) & = 4 \cdot \prd_{[F:K]=2} \ffrac{\order \mu(F)}{2}\,\ \ 
\hbox{(see \eqref{analytic})}.
\end{aligned}\right .
\end{equation}

For instance,  $w_2(\Q) = 24$  and $\zeta_\Q^{}(-1) = \ffrac{1}{12}$; then
$\order \BK_2 (\Z) = 2$, giving $\order \BR_2 (\Z) = 1$ as expected since 
$\CT_{\Q, p} = 1$ for all $p$.

\smallskip
Thus, the computation of  the $p$-Sylow subgroup
$\order \BK_2(\BZ_K) [p^\infty]$ of $\BK_2 (\BZ_K)$ is possible for $p=2$ 
and $p=3$ with the following formulas (using $\zeta_\Q^{}(-1) = \ffrac{1}{12}$ 
and \eqref{flp} for $m=2$):
\begin{equation*}
\left\{\begin{aligned}
\order \BK_2 (\BZ_K)[2^\infty] & \sim \ffrac{w_2(K)}{12} \times\!\!
\prd_{\chi \ne \1} \ffrac{1}{1-2 \chi(2)}\,\BL_2(-1, \chi), \\
\order \BK_2 (\BZ_K) [3^\infty] & \sim \ffrac{w_2(K)}{12} \times\!\!
\prd_{\chi \ne \1} \ffrac{1}{1-3 \chi(3)}\,\BL_3(-1, \chi), 
\end{aligned}\right .
\end{equation*}
where $\sim$ means equality up to a $p$-adic unit factor for the considered$p$. 
Since the denominators $1-p\,\chi(p)$ are invertible in $\Z_p$, the formulas 
become with expression \eqref{BT} of $w_2(K)$:
\begin{equation*}
\left\{\begin{aligned}
\order \BK_2 (\BZ_K) [2^\infty] & \sim \prd_{[F:K]=2} \ffrac{ \order \mu_2(F) }{2} 
\times\!\! \prd_{\chi \ne \1}\,\BL_2(-1, \chi), \\
\order \BK_2(\BZ_K) [3^\infty] & \sim \ffrac{1}{3} \prd_{[F:K]=2} \order \mu_3(F)
\times\!\! \prd_{\chi \ne \1}\,\BL_3(-1, \chi).
\end{aligned}\right .
\end{equation*}

\subsection{Formulas for quadratic and cubic fields}\label{formulas}
We deduce the following expressions with $p \in \{2, 3\}$;
due to exact sequence \eqref{regular} when $p=2$, we do not write the 
corresponding results for $\BR_2(\BZ_K)$.

\smallskip
({\bf a}) Real quadratic fields.

\smallskip
\quad (i) For real quadratic fields $K = \Q(\sqrt m)$, $m \ne 2, 3$, $p=2$: 
\begin{equation*}
\left\{\begin{aligned}
\prd_{[F:K] = 2} \ffrac{1}{2} \cdot \order \mu_2(F) & = 2
\hbox{ ($F=K(\mu_4)$, $F=K(\mu_6)$)}; \\
\order \BK_2(\BZ_K) [2^\infty] & \sim 2 \cdot  \BL_2(-1, \chi).
\end{aligned}\right .
\end{equation*}

\quad (ii) For $\Q(\sqrt 2)$, $\prd_{[F:K]=2} \ffrac{\order \mu_2(F)}{2}  = 4$
($F=K(\mu_8)$, $F=K(\mu_6)$); so, 
$\order \BK_2\BZ_{\Q(\sqrt 2)} [2^\infty] \sim 4 \cdot \BL_2(-1, \chi) \sim 4$.

\smallskip\smallskip
\quad (iii) For $\Q(\sqrt 3)$, $\prd_{[F:K]=2} \ffrac{\order \mu_2(F)}{2}  = 4$ 
($F=K(\mu_4)$, $F=K(\mu_6) = K(\mu_4)$); so, 
$\order \BK_2\BZ_{\Q(\sqrt 3)} [2^\infty] \sim 4 \cdot \BL_2(-1, \chi) \sim 8$.

\smallskip
\quad (iv) For $p=3$, $K \ne \Q(\sqrt 3)$, the formulas are:
\begin{equation*}
\left\{\begin{aligned}
\ffrac{1}{3}  \prd_{[F:K] = 2}  \order \mu_3(F) & = 1 \hbox{ ($F=K(\mu_6)$)}; \\
\order \BK_2(\BZ_K) [3^\infty] & \sim \BL_3(-1, \chi). 
\end{aligned}\right .
\end{equation*}

\smallskip
\quad (v) For $\Q(\sqrt 3)$, $\ffrac{1}{3} \prd_{[F:K]=2} \order \mu_3(F) = 3$ 
($F=K(\mu_3)$, $F=K(\mu_4) = K(\mu_6)$); so, 
$\order \BK_2\BZ_{\Q(\sqrt 3)} [3^\infty] \sim 3 \cdot \BL_3(-1, \chi) \sim 3$.

\medskip
({\bf b}) Cyclic cubic fields.

\smallskip
\quad (i) For cubic fields $K$ of conductor distinct from $9$, $p=3$:
\begin{equation*}
\left\{\begin{aligned}
\ffrac{1}{3} \prd_{[F:K]=2} \order \mu_3(F) & = 1 
\hbox{ ($F=K(\mu_6)$)}; \\
\order \BK_2 (\BZ_K) [3^\infty] & \sim \BL_3(-1, \chi) \times \BL_3(-1, \chi^2), 
\end{aligned}\right .
\end{equation*}
for the two conjugate characters of order $3$ of $K$.

\smallskip
\quad (ii) For the cubic field of conductor $9$, 
$\ffrac{1}{3} \prd_{[F:K]=2} \order \mu_3(F) = 3$ ($F=K(\mu_9)$);
then $\order \BK_2(\BZ_K) [3^\infty] = 3 \cdot \BL_2(-1, \chi) \sim 1$.

\smallskip
\quad (iii)  For $p=2$, the formulas are, for the two conjugate characters of order 
$3$ of $K$:
\begin{equation*}
\left\{\begin{aligned}
\prd_{[F:K]=2} \ffrac{1}{2}\cdot \order \mu_2(F) & = 2
\hbox{ ($F=K(\mu_4)$, $F=K(\mu_6)$)}; \\
\order \BK_2 (\BZ_K) [2^\infty] & \sim 2 \cdot \BL_2(-1, \chi) \times \BL_2(-1, \chi^2).
\end{aligned}\right .
\end{equation*}

\section{Definition of a \texorpdfstring{$p$}{Lg}-adic pseudo-measure
 \texorpdfstring{$(\CS_{\!L_n})_n$}{Lg}}
Let $K$ be a real abelian field and put $L_n := K(\mu_{qp^n})$, $q \in \{p, 4\}$, 
as usual, and $n \geq 0$.

\subsection{The Stickelberger elements}
The conductor of $L_n$ is of the form $f_{L_n} = q p^n f$, for a prime-to-$p$
integer $f$, taking $n$ large enough if $p$ ramifies in $K$ (otherwise, $f$ is 
the conductor of $K$) (in formulas we shall abbreviate $f_{L_n}$ by $f_n$). 

\smallskip
Then put, where all Artin symbols are taken over $\Q$:

\centerline{$\CS_{\!L_n} := -\sm_{a=1}^{f_n} \Big(\ffrac{a}{f_n}-\ffrac{1}{2} \Big)
\,\Big(\ffrac{L_n}{a} \Big)^{-1}$, }

\noindent
as restriction to $L_n$ of $\CS_{\Q(\mu_{f_n}^{})}$, 
where $a$ runs trough the prime-to-$f_n$ integers $a \in [1, f_n]$.

\subsection{Norms of the Stickelberger elements}
Let $f$ and $m$ be such that $m \mid f$; consider $\Q(\mu_f^{})$ and $\Q(\mu_m)$
and let $\Norm_{\Q(\mu_f^{})/\Q(\mu_m)}$ be the restriction map
$\Q[\Gal(\Q(\mu_f^{}) / \Q)] \to \Q[\Gal(\Q(\mu_m)/\Q)]$.
We have: 
\begin{equation}\label{norms}
\Norm_{\Q(\mu_f^{})/\Q(\mu_m)} (\CS_{\Q(\mu_f^{})}) = 
\prd_{\ell \mid f,\,\ell \nmid m} \!\! \Big(1-\Big(\ffrac{\Q(\mu_m)}{\ell} \Big)^{-1}\Big) 
\cdot \CS_{\Q(\mu_m)}. 
\end{equation}

Let $L/K$ be an extension of abelian fields of conductors $m$ and $f$;
we define $\CS_L:= \Norm_{\Q(\mu_f^{})/L}(\CS_{\Q(\mu_f^{})})$ and 
$\CS_K:= \Norm_{\Q(\mu_m^{})/K}(\CS_{\Q(\mu_m^{})})$, respectively; 
then, $\Norm_{L/K}(\CS_{\!L})  = \prd_{\ell \mid f,\  \ell \nmid m}\!\!
\Big(1-\Big(\ffrac{K}{\ell} \Big)^{-1}\Big) \cdot \CS_K$. This implies 
$\Norm_{L/K}(\CS_L) = 0$ as soon as a prime $\ell \mid f$ totally splits in $K$.

\subsection{Twists of the Stickelberger elements}
Let $c$ be an integer prime to $2 p f$; for $L_n = K(\mu_{qp^n}^{})$, put:
$\CS_{\!L_n}(c) := \Big(1-c\Big(\ffrac{L_n}{c} \Big)^{-1} \Big)\cdot \CS_{\!L_n}$.
Then $\CS_{\!L_n}(c) \in \Z[\Gal(L_n/\Q)]$. Indeed, we have:
$$\CS_{\!L_n} (c)=\ffrac{-1}{f_n} \sm_a 
\Big[a \Big(\ffrac{L_n}{a} \Big)^{-1}-a c \Big(\ffrac{L_n}{a} \Big)^{-1}\Big(\ffrac{L_n}{c} \Big)^{-1}\Big]
+ \ffrac{1-c}{2} \sm_a \Big(\ffrac{L_n}{a} \Big)^{-1}; $$
let $a'_{c} \in [1, f_n]$ be the unique integer such that 
$a'_{c} \cdot c \equiv a \pmod {f_n}$;
put $a'_{c} \cdot  c-a = \lambda^n_a(c) f_n$, 
$\lambda^n_a(c) \in \Z$; then, using the bijection 
$a \mapsto a'_{c}$ in the second summation and the fact that
$\Big(\ffrac{L_n}{a'_{c}} \Big) \Big(\ffrac{L_n}{c} \Big) =  \Big(\ffrac{L_n}{a} \Big)$:
\begin{equation}\label{twist}
\left\{\begin{aligned}
 \CS_{\!L_n}(c) & = \ffrac{-1}{f_n}  \Big[ \sm_a a \Big(\ffrac{L_n}{a} \Big)^{-1}\!\!-
\sm_a a'_{c}  c \Big(\ffrac{L_n}{a'_{c}} \Big)^{-1} \! \Big(\ffrac{L_n}{c} \Big)^{-1}\Big] \\
&\hspace{4.1cm}+\ffrac{1-c}{2} \sm_a \Big(\ffrac{L_n}{a} \Big)^{-1} \\
& =  \sm_a \Big( \lambda^n_a(c)+\ffrac{1-c}{2} \Big) \Big(\ffrac{L_n}{a} \Big)^{-1}
\!\! \in \Z[\Gal(L_n/\Q)].
\end{aligned}\right.
\end{equation}

This twisted form gives the Stickelberger $p$-adic measure used to generate the  
$\BL_p$-functions of $K$ as explained in the next Section.

\section{The measure \texorpdfstring{$(\CS^*_{\!L_n}(c))_n$}{Lg}
defining \texorpdfstring{$\BL_p(s, \chi)$}{Lg}} 

Consider the algebras $A_n := \Z_p / qp^n\Z_p[\Gal(L_n/\Q)],\ n \geq 0$. 
The Mellin transform (e.g., \cite[\S\,II.1]{Gr1}) is defined, on $A_n$, 
by the following image of any $\sigma \in G$, where $a$, 
defined modulo $f_n = qp^n f$, represents $\sigma$ as Artin symbol:
\begin{equation*}
\left \{ \begin{aligned}
\sigma & = \Big(\ffrac{L_n}{a} \Big) \mapsto
\theta(a) \langle a \rangle^s \Big(\ffrac{L_n}{a} \Big)^{-1} 
= a \langle a \rangle^{s-1} \Big(\ffrac{L_n}{a} \Big)^{-1} \\
& = a \langle a \rangle^{s-1} \sigma^{-1} \pmod {qp^n f},\ \,s \in \Z_p, 
\end{aligned} \right.
\end{equation*}
and the image, in $A_n$, of the expression \eqref{twist} of $\CS_{\!L_n}(c)$
by this transform
yields the $p$-adic measure in $\ds \varprojlim_n A_n$ giving annihilation 
theorems for real invariants and $p$-adic functions $\BL_p(s, \chi)$ for $n \to \infty$;
indeed, since $L_n = K(\mu_{qp^n})$, one gets, from the norm relations \eqref{norms}, 
$$\Norm_{L_{n+1}/L_n} (\CS^*_{\!L_{n+1}}(c)(s)) \equiv \CS^*_{\!L_n}(c)(s)
\pmod{qp^n},\ \,n \geq 0. $$
So, we will use the following approximations: 
\begin{equation*}
\CS^*_{\!L_n}(c)(s) \equiv \sm_{a=1}^{f_n} \Big(\lambda^n_a(c) 
+ \ffrac{1-c}{2} \Big)\,a^{-1} \langle a \rangle^{1-s} \Big(\ffrac{L_n}{a} \Big)\!\!
\pmod{qp^n}.
\end{equation*}

\subsection{Approximations modulo \texorpdfstring{$q p^n$}{Lg}
of \texorpdfstring{$\BL_p(s, \chi)$}{Lg}}

The $\BL_p$-functions, of even Dirichlet's characters, are obtained by 
means of the Mellin transform of the previous twisted pseudo-measures, 
so that the Mellin transform of $\Big(1-c\Big(\ffrac{L_n}{c} \Big)^{-1} \Big)$ is 
$\Big(1- \langle c \rangle^{1-s}  \Big(\ffrac{L_n}{c} \Big)\Big)$ and is a
factor of the $p$-adic measure $\big(\CS^*_{\!L_n}(c)(s) \big)_n$ that must be dropped
when this makes sense (thus in any case, except $\chi = \1$ and $s=1$):

\begin{theorem}\label{Lp}
 Let $c$ be any odd integer, prime to $p$ and to the conductor $f$
of $K$. For all $n$ large enough, let $f_n$ be the conductor of $L_n = K(\mu_{qp^n})$, 
and for all $a \in [1, f_n]$, prime to $f_n$, let $a'_{c}$ be the unique integer in $[1, f_n]$ 
such that $a'_{c} \cdot  c-a = \lambda^n_a(c) f_n$, $\lambda^n_a(c) \in \Z$.
The twisted $p$-adic measure generating the $\BL_p$-functions is given by
the restriction $\CS^*_{K, n}(c)(s)$ of $\CS^*_{\!L_n}(c)(s)$ to $K$, giving, for
all $n \geq 0$:

\centerline{$\CS^*_{K, n}(c)(s) \equiv  \sm_{a=1}^{f_n} \Big[ \lambda^n_a(c)+\ffrac{1-c}{2}\Big]
\,a^{-1}\langle a \rangle^{1-s} \Big(\ffrac{K}{a} \Big) \pmod{qp^n}$.} 

\smallskip\noindent
Whence the $\BL_p$-functions of Dirichlet's character $\chi$ of $K$:

\centerline{$\BL_p(s, \chi) = \ffrac{1}{1-\chi(c) \langle c \rangle^{1-s}} \times {\ds \lim_{n \to \infty}} 
 \sm_{a = 1}^{f_n} \big [\lambda_a^n(c) +\ffrac{1-c}{2} \big]\,a^{-1} \langle a \rangle^{1-s} \chi(a)$.} 
\end{theorem}

This formula is used in Appendices \ref{B1}, \ref{B2}, \ref{C1}, \ref{C2}, \ref{D3}, 
for explicit numerical computations.

\begin{corollary}\label{s=-1}
Assume that $p \in \{2, 3\}$; in this context, $\theta^2 = \1$.
Since for $s=-1$, $a^{-1} \langle a \rangle^{1-s} = a^{-1} \langle a \rangle^2 = a$, 
we get:

\centerline{$\BL_p(-1, \chi) = \ffrac{1}{1-\chi(c) \langle c \rangle^2} \times {\ds \lim_{n \to \infty}} 
 \sm_{a = 1}^{f_n} \big [\lambda_a^n(c) +\ffrac{1-c}{2} \big]\,a \chi(a)$.}
\end{corollary}

\begin{corollary}\label{formulasbis}
Assume $\chi \ne \1$ and $\chi(c) \ne 1$. For $s=-1$, $\langle c \rangle 
\equiv 1 \pmod q$ and $\chi(c) \in \mu_{[K:\Q]}^{} \setminus \{1\}$; 
so, in the quadratic case with $p=3$ and in the cubic case with $p=2$, 
$1-\chi(c) \langle c \rangle^2$ is invertible. Then:

\smallskip
$\bullet$ $1-\chi(c) \langle c \rangle^2 \sim 2$ for the quadratic case and $p=2$, 

\smallskip
$\bullet$ $1-\chi(c) \langle c \rangle^2\sim 1-\zeta_3$ for the cubic case and $p=3$.

\smallskip\noindent
From formulas of Section \ref{formulas} one gets:

\smallskip
(i) For $p=2$ in the quadratic case $K=\Q(\sqrt m)$, $m \ne 2, 3$ and a half summation 
giving $\frac{1}{2}\,\BL_2(s, \chi)$,  for 
$\order \BR_2(\BZ_K)[2^\infty] \!=\! \ffrac{1}{4} \order \BK_2(\BZ_K)[2^\infty]$
one obtains (since $1-\chi(c) \langle c \rangle^2 \sim 2$):
\begin{equation*}
\left \{\begin{aligned}
\order \BR_2(\BZ_K)[2^\infty] & \sim \ffrac{1}{2}\,\BL_2(-1, \chi) \\
& \equiv \ffrac{1}{2}\! \sm_{a = 1}^{f_n/2} \big [\lambda_a^n(c) 
+ \ffrac{1-c}{2} \big] a \chi(a) \pmod {2^n}.
\end{aligned}\right.
\end{equation*}

\smallskip
(ii) For the cubic case with $p=3$, 
$(1-\chi(c) \langle c \rangle^2)(1-\chi^2 (c) \langle c \rangle^2) 
\sim (1 - \zeta_3)( 1 - \zeta_3^2) \sim 3$, giving, with $\Norm = \Norm_{\Q(\mu_3)/\Q}$:
\begin{equation*}
\left \{\begin{aligned}
\order \BK_2(\BZ_K)[3^\infty] & \sim \BL_3(-1, \chi) \cdot \BL_3(-1, \chi^2) \\
& \equiv  \ffrac{1}{3}\,\Norm \Big( \sm_{a = 1}^{f_n/2} \big [\lambda_a^n(c) 
+ \ffrac{1-c}{2} \big] a \chi(a) \Big) \!\! \pmod {3^n}.
\end{aligned}\right.
\end{equation*}
\end{corollary}

\subsection{Genus theory of \texorpdfstring{$p$}{Lg}-adic pseudo-measures}
We have obtained in \cite[Th\'eor\`eme 0.2]{Gr34}, the following results
which does not seem very known but states the property of ``stability''
mentioned in the Introduction and may be linked to the existence (or not) of 
zeroes of $\BL_p$-functions (for examples about zeroes of $\BL_p$-functions, 
one may refer to \cite{W1, W2, W3, W4, W} and see Appendix \ref{C3} for some 
illustrations of this influence):

\begin{theorem} \label{thmfond}
Under Notations \ref{notations}, where $\chi =  \varphi \psi$ is even, $\varphi$ 
of prime-to-$p$ order and $\psi$ of order $p^e$, $e \geq 1$, we have, with 
$\BD = \sum_{\ell} p^{d_\ell}$ and $\epsilon = 0$ (resp. 
$\epsilon = 1$) when $ \varphi \ne \1$ (resp. $\varphi =  \1$):

\smallskip
{\bf (a)}  $v_{\mathfrak m} \Big(\frac{1}{2} \BL_p(s, \chi) \Big) \geq \BC(s),\,\forall\,s \in \Z_p$, 
where $\BC(s)$ is as follows:

\smallskip
\quad ($i_p$) $p \ne 2$ and $\theta^{-1} \varphi (p) \ne 1$ (i.e., $p$ not totally split in $M_0$); then:
$$\BC(s) = \BC = \BD-\epsilon. $$

\smallskip
\quad ($i_2$) $p = 2$ and $ \varphi(2) \ne 1$  (i.e., $2$ not totally split in $M_0$); then:
$$\BC(s) = \BC=\BD. $$

\smallskip
\quad ($ii_p$) $p \ne 2$, $\theta^{-1} \varphi (p) = 1$ and 
$\theta^{-1} \chi (p) \ne 1$ (i.e., $p$ totally split in $M_0$ and not totally split in $M$); then:
$$\BC(s) = \BC=\BD+p^{d_p}. $$

\smallskip
\quad ($ii_2$) $p = 2$, $ \varphi(2) = 1$ and $\theta^{-1} \chi (2) \ne 1$ (i.e., $2$ 
totally split in $M_0$ and not totally split in $M$); then:
$$\BC(s) = \BC=\BD+2^{d_2}-2 \epsilon. $$

\smallskip
\quad ($iii_p$) $p \ne 2$, $\theta^{-1} \chi (p) = 1$ (i.e., $p$ totally split in $M$); then:
$$\BC(s) :=\BD+ v_{\mathfrak m}(p\,s). $$

\smallskip
\quad ($iii_2$) $p = 2$, $\theta^{-1} \chi (2) = 1$ (i.e., $2$ totally split in $M$); then:
$$\BC(s) :=\BD+ v_{\mathfrak m}(q\,s)-2 \epsilon. $$

\smallskip
{\bf (b)}  We have either the equality:
$$v_{\mathfrak m} \big(\ffrac{1}{2} \BL_p(s, \chi) \big) = \BC(s),\,\forall\,s \in \Z_p, $$ 
or the strict inequality:
$$v_{\mathfrak m} \big(\ffrac{1}{2} \BL_p(s, \chi) \big) > 
\BC(s),\,\forall\,s \in \Z_p, $$
(with $s \ne 0$ in cases ($iii_p$) and ($iii_2$)).
\end{theorem}

So, under the case of stability $v_{\mathfrak m} \big(\frac{1}{2} \BL_p(s, \chi) \big) 
= \BC(s),\,\forall\,s \in \Z_p$, a computation may be done at a suitable value 
$s$ for $\BL_p(s, \chi)$, e.g., $s=1$ giving $\order \CT_{K, p} \sim 
[K \cap\,\Q_\infty : \Q] \times \prd_{\chi \ne 1} \ffrac{1}{2}\,\BL_p (1, \chi)$;
thus $\order \CT_{K, 2} \sim \ffrac{1}{2}\,\BL_2 (1, \chi)$ for $K$ real quadratic, 
except $K$ of conductor $8$, or $\order \CT_{K, 3} \sim \BL_3 (1, \chi) 
\times \BL_3 (1, \chi^2)$ for $K$ cyclic cubic, except $K$ of conductor $9$.

\smallskip
Which gives the main process and Theorem \ref{main}\,{\bf A}, 
where $\BC(s) = \BC$, as soon as $s \in \Z_p^\times$, 
since $v_{\mathfrak m}(s) = 0$ in cases (iii) (e.g., $s = \pm 1$):

\begin{theorem} \label{thmfinal}
We obtain the following results:

\smallskip
(i) Let $K$ be a real quadratic field of character $\chi$ of conductor $\ne 8$, and 
set $p=2$. Assume that the knowledge of $\order \CT_{K, 2}$ (e.g., using formula 
\eqref{T=HRW}) implies $v_2 \big(\frac{1}{2} \BL_2(1, \chi) \big) = \BC$; then 
$v_2 \big(\frac{1}{2} \BL_2(s, \chi) \big) = \BC,\,\forall\,s \in \Z_2$, whence
in particular, $\order \BK_2(\BZ_K)[2^\infty] = 2^{\hbox{\tiny $\BC$}+2}$.

\smallskip
We have $\BC = \BD$ (resp. $\BC = \BD-1$) if $2$ splits (resp. does not split) 
in $M = \Q(\sqrt {-m})$.

\smallskip
(ii)  Let $K$ be a cyclic cubic field of character $\chi$ of conductor $\ne 9$, 
and set $p=3$. Assume that the knowledge of $\order \CT_{K, 3}$ implies 
$v_{\mathfrak m} \big(\BL_3(1, \chi) \big) = \BC$; then $v_{\mathfrak m} 
\big(\BL_3(s, \chi) \big) = \BC,\,\forall\,s \in \Z_3$, then 
in particular, $\order \BK_2(\BZ_K)[3^\infty] = 3^{\hbox{\tiny $\BC$}}$.
We have $\BC = \BD-1$.

\smallskip
(iii) If $v_2 \big(\frac{1}{2} \BL_2(1, \chi) \big) > \BC$ 
(resp. $v_{\mathfrak m} \big(\BL_3(1, \chi) \big) > \BC$), then one obtains
$\order \BK_2(\BZ_K)[2^\infty] > 2^{\hbox{\tiny $\BC$}+2}$ (resp. 
$\order \BK_2(\BZ_K)[3^\infty] > 3^{\hbox{\tiny $\BC$}}$).
\end{theorem}

Considering for instance \cite[Theorem 1.2]{DL} giving, for this specific family of
real quadratic fields: 
$$\BK_2(\BZ_K)[2^\infty] \simeq (\Z/2\Z)^{n-1} \times \Z/2^\delta\Z,\ \,\delta \geq 3, $$ 
the Deng--Li conjecture is $\delta = 3$; we intend to prove it by means of
the previous process from showing that $\order \CT_{K, 2} = 2^{n}$:

\begin{corollary}\label{remfond}
Under the conditions: $\order \CH_{K, 2} = 2^{n-1}$, $\CR_{K, 2}=1$
and $\order \CW_{K, 2} = 2$, then $\order \CT_{K, 2} = 2^n$, 
giving the valuation $n$ for $\frac{1}{2} \BL_2(1, \chi)$; so, Theorem \ref{thmfinal}\,(i) 
yields $\BC = \BD = n$, since $2$ splits in $\Q(\sqrt {-m})$, and the conjecture 
$\delta = 3$ follows.
\end{corollary}

\subsection{Modulus of continuity of \texorpdfstring{$\BL_p(s, \chi)$}{Lg}}
Whatever $p$, if $\chi$ is of prime-to-$p$ order, genus 
theory is empty and this raise the question of the ``independence'' 
(or not) of $\order \BK_2 (\BZ_K)[p^\infty]$ and $\order \CT_{K, p}$.
This is related to the rank formula 
\eqref{equalranks} when $K$ contains the maximal real subfield of 
$\Q(\mu_p)$, but is also a consequence of the existence of a
non-trivial modulus of continuity for $\BL_p(s, \chi)$, whatever 
$\chi = \varphi \psi$, as follows \cite[Th\'eor\`eme 0.3]{Gr34}:

\begin{theorem}
With Notations \ref{notations}, we have, for all $s, t \in \Z_p$:
$$\ffrac{1}{2}\BL_p(t, \chi)-\ffrac{1}{2}\BL_p(s, \chi) \equiv a\,(t-s)
\pmod {q\,{\mathfrak m}^{\!\hbox{\tiny $\BV$}} (t-s)}, $$
with a computable constant $a$ ($a = 0$ when $\varphi \ne \1$) and $\BV$ of 
the form $\BV = \ds \BD-\max_\ell (p^{d_\ell}-p^{d_p}+\epsilon, \epsilon)$
($\epsilon = 1$ if $\varphi = \1$, $0$ otherwise), 
where $\ell$ runs trough the set of primes ramified in $M/M_0$, totally 
split in $M_0/\Q$ and such that $\frac{1}{q} \log_p(\ell) \not\equiv 0 \pmod p$.
\end{theorem}

In the case $\psi = \1$, $\varphi \ne \1$, the above formula becomes:
$$\ffrac{1}{2}\BL_p(t, \varphi) \equiv \ffrac{1}{2}\BL_p(s, \varphi) \pmod {q\,(t-s)}. $$

Thus, even if the genus principle is empty, there is a non-trivial congruence
between the orders of the two corresponding invariants.

\smallskip
For instance, we get $\ffrac{1}{2}\BL_p(1, \varphi) \equiv
\ffrac{1}{2}\BL_p(-1, \varphi) \pmod {2q}$, whence:
\begin{equation*}
\left\{\begin{aligned}
\order \BK_2(\BZ_K)[3^\infty] & \equiv \order \CT_{K, 3} \pmod 3,\ \ 
\hbox{for quadratic fields},  \\
\order \BR_2(\BZ_K)[2^\infty] & \equiv \order \CT_{K, 2} \pmod 8,\ \  
\hbox{for cyclic cubic fields}.
\end{aligned}\right.
\end{equation*}

In the case of real quadratic fields for $p=3$ (resp. of cyclic cubic fields for $p=2$), see 
the numerical results given in Appendix \ref{B2} (resp. Appendix \ref{C2}).

\section{Proof of the Deng--Li conjecture \texorpdfstring{$\delta = 3$}{Lg}}
\label{DengLi}
We consider the family defined in \cite[Theorem 1.2]{DL}. We will prove the 
conjecture $\delta = 3$ in the writing:
$$\BK_2(\BZ_K)[2^\infty] \simeq (\Z/2\Z)^{n-1} 
\times \Z/2^\delta \Z$$ 
and some other properties of this family of real quadratic fields.

\smallskip
Recall that $K = \Q(\sqrt m)$, $m= \ell_1 \ell_2 \cdots \ell_n$, $n$ even, with:

\smallskip
(i) $\ell_1 \equiv 3 \pmod 8$, $\ell_i \equiv 5 \pmod 8$ for $i \geq 2$, 

\smallskip
(ii) $\big( \frac{\ell_1}{\ell_2} \big) = -1$,  $\big( \frac{\ell_1}{\ell_j} \big) = 1$ for $j \geq 3$, 

\smallskip
(iii) $\big( \frac{\ell_i}{\ell_j} \big) = -1$ for $2 \leq i < j \leq n$.

\smallskip
Condition (i) implies $m \equiv -1 \pmod 8$ and $K$ of discriminant $4m$; then
$\big( \frac{\ell_i}{\ell_j} \big) = \big( \frac{\ell_j}{\ell_i} \big)$ for all $i, j$ since
for $i \ne j$, one of the two primes is congruent to $1$ modulo $4$.
It implies also that the fundamental unit $\varepsilon$ is of norm $1$ since $-1$
is not norm in $K/\Q$; then, if $\varepsilon = a+b \sqrt m$, one has $a^2+b^2 
\equiv 1 \pmod 8$, whence, either $a = 4 a'$ with $b$ odd, or $b = 4b'$ with $a$ odd.

\smallskip
Let ${\mathfrak p}$ be the prime ideal above $(2)$ and let $K_ {\mathfrak p}$
be the completion of $K$ at ${\mathfrak p}$; thus $K_ {\mathfrak p} = \Q_2(\sqrt {-1})$
proving that $2$ is local norm at $2$ in $K/\Q$ and that the norm 
group of local units is equal to $1+4 \Z_2$. 

\smallskip
The Hasse norm residue 
symbols, of the form $\Big(\frac{\ell_i,\  K/\Q}{\ell_j} \Big)$, characterizing 
the property ``$\ell_i$ local norm at $\ell_j$ in $K/\Q$'', are given by the 
quadratic residue symbols $\big(\frac{\ell_i}{\ell_j} \big)$; indeed, $\ell_i$
norm in $\Q_{\ell_j}(\sqrt m)/\Q_{\ell_j}$ is equivalent (for $i \ne j$) to
$\ell_i$ square in $\Q_{\ell_j}^\times$ since the norm group of local 
units of $\Q_{\ell_j}(\sqrt m)$ is of index $2$ in $\mu_{\ell_j-1}^{} 
\oplus (1+\Z_{\ell_j}^{})$.

\smallskip
We can add the properties:

(iv) $\big( \frac{2}{\ell_i} \big) = \big(-1\big)^{\frac{\ell_i^2-1}{8}} = -1$, 
for all $i$, $\big( \frac{\ell_1}{2} \big) = -1$, $\big( \frac{\ell_i}{2} \big) = 1$ 
for $i \geq 2$ (in the meaning $\ell_i$ norm (or not) in $\Q_2(\sqrt{-1})/\Q_2$).

\smallskip
(v) $\big( \frac{\ell_i}{\ell_i} \big) = 1$ for all $i \ne 2$ and $\big( \frac{\ell_2}{\ell_2} \big) = -1$, 
obtained from the product formula of the Hasse norm residue symbols of fixed $\ell_i$ with 
$n$ even. 

\medskip
Put $G := \Gal(K/\Q) =: \langle \sigma \rangle$ and: 
$$\Omega_K := \{(s_1, \ldots, s_{n+1}) 
\in G^{n+1},\ s_1 \cdots s_{n+1} =1 \} \simeq G^n; $$ 
let $h_K$ be the map $\Q^\times \to \Omega_K$ defined, by 
means of the Hasse norm residue symbols, by
$h_K(x) = \Big(\Big(\frac{x,\  K/\Q}{v} \Big) \Big)_{\! v}$, 
where $v$ runs trough the $n+1$ places ramified in $K/\Q$, 
and for which the product formula:
$$\prd_{v\ {\rm ramified}}\Big(\ffrac{x,\  K/\Q}{v} \Big) = 1$$ 
holds as soon as $(x)$ is the norm of an ideal in $K/\Q$, hence local 
norm at every non-ramified place (i.e., $x \in \{-1, 1\} \cdot \Q^{\times 2}$).

Whence the matrix of symbols, product formula taken on each line:
\begin{equation}\label{array}
\begin{tabular}{|l|l|l|l|c|c|c|c|c|c|c}
\hline
\ft${}$\ns & \ft$2$\ns  & \ft$\ell_1$\ns & \ft$\ell_2$\ns & \ft$\ell_3$\ns & \ft$\ldots$\ns 
& \ft$ \ell_i$\ns & \ft$\ldots$\ns & \ft$\ell_n$\ns  \\  
\hline
\ft$2$\ns & \ft$1$\ns  & \ft$-1$\ns & \ft$-1$\ns & \ft$-1$\ns & \ft$\ldots$\ns 
& \ft$-1$\ns & \ft$\ldots$\ns & \ft$-1$\ns  \\  
\hline
\ft$\ell_1$\ns & \ft$-1$\ns  & \ft$1$\ns & \ft$-1$\ns & \ft$ 1$\ns & \ft$\ldots$\ns 
& \ft$1$\ns & \ft$\ldots$\ns & \ft$1$\ns  \\  
\hline 
\ft$\ell_2$\ns & \ft$1$\ns  & \ft$-1$\ns & \ft$-1$\ns & \ft$-1$\ns & \ft$\ldots$\ns 
& \ft$-1$\ns & \ft$\ldots$\ns & \ft$-1$\ns  \\ 
\hline 
\ft$\ell_3$\ns & \ft$1$\ns  & \ft$1$\ns & \ft$-1$\ns & \ft$1$\ns & \ft$\ldots$\ns 
& \ft$-1$\ns & \ft$\ldots$\ns & \ft$-1$\ns  \\ 
\hline 
\ft${\vdots}$\ns & \ft${}$\ns  & \ft${}$\ns & \ft${}$\ns & \ft${}$\ns & \ft${}$\ns 
& \ft${}$\ns & \ft${}$\ns & \ft${}$\ns   \\ 
\hline
\ft$\ell_i$\ns & \ft$1$\ns  & \ft$1$\ns & \ft$-1$\ns & \ft$-1$\ns & \ft$\ldots$\ns 
& \ft$1$\ns & \ft$\ldots$\ns & \ft$-1$\ns   \\   
\hline 
\ft${\vdots}$\ns & \ft${}$\ns  & \ft${}$\ns & \ft${}$\ns & \ft${}$\ns & \ft${}$\ns 
& \ft${}$\ns & \ft${}$\ns & \ft${}$\ns  \\ 
\hline
\ft$\ell_n$\ns & \ft$1$\ns  & \ft$1$\ns & \ft$-1$\ns & \ft$-1$\ns & \ft$\ldots$\ns 
& \ft$-1$\ns & \ft$\ldots$\ns & \ft$1$\ns  \\  
\hline 
\ft$-1$\ns & \ft$-1$\ns  & \ft$-1$\ns & \ft$1$\ns & \ft$1$\ns & \ft$\ldots$\ns 
& \ft$1$\ns & \ft$\ldots$\ns & \ft$1$\ns   \\ 
\hline
\end{tabular}
\end{equation}

\subsection{Proof of \texorpdfstring{$\CH_{K, 2} \simeq (\Z/2\Z)^{n-1}$}{Lg}}
This is proven in \cite{DL}, but we can bring more informations and remarks.
The first one is given by the Chevalley--Herbrand formula \cite[pp. 402-406]{Che}
in $K/\Q$:
$$\order \CH_{K, 2}^G = \frac{2^{n+1}}
{[K : \Q] (E_\Q : E_\Q \cap \Norm_{K/\Q} (K^\times))}, $$ 

\noindent
since $n+1$ primes ramify, with $E_\Q = \{\pm 1\}$; we shall write instead:
\begin{equation}
\order \CH_{K, 2}^G = \frac{\order \Omega_K}{\order h_K (E_\Q)} =
\frac{2^n}{\order h_K (\{\pm 1\})} = 2^{n-1}, 
\end{equation}

\noindent
since $-1 \not\in \Norm_{K/\Q} (K^\times)$. 

\begin{lemma}\label{lem1}
The subgroup $\CH_{K, 2}^G$ is elementary of $2$-rank $n-1$; it is generated
by the classes of the prime ideals ${\mathfrak l}_i \mid \ell_i$, $i = 1, \ldots , n$, 
and by the class of ${\mathfrak p} \mid (2)$. There are two independent
relations of principality of the form ${\mathfrak p}^{a_0} \prod_{i=1}^n 
{\mathfrak l}_i^{a_i} = (\alpha)$ between the ramified primes (where the 
exponents are $0$ or $1$), an obvious one being 
$\prod_{i=1}^n {\mathfrak l}_i = (\sqrt m)$.
\end{lemma}

\begin{proof}
We have the classical exact sequence, where $E_K = 
\langle -1, \varepsilon \rangle$ is the group of units of $K$ and
where $\CH_{K, 2}^\ram$ is the subgroup of $\CH_{K, 2}^G$ generated 
by the classes of the ramified primes:
$$1 \to \CH_{K, 2}^\ram \too \CH_{K, 2}^G \too 
\{\pm 1\} \cap \Norm_{K/\Q} (K^\times)/ \Norm_{K/\Q}(E_K) \to 1, $$
giving here $\CH_{K, 2}^G  = \CH_{K, 2}^\ram \simeq (\Z/2\Z)^{n-1}$, the
right term being trivial; so, there are exactly two independent relations of 
principality between the ramified primes.

\smallskip\noindent
Then, from \cite[\S\,4.4]{Gr7} generalizing our old papers in ``Annales 
de l'Institut Fourier'', we have for the second element of the filtration:
$$\order \Big( \big (\CH_{K, 2}/\CH_{K, 2}^G \big)^G \Big) = \frac{\order \Omega_K}
{\order h_K (\Lambda)},\ \,\Lambda = \langle -1, 2, \ell_1, \ldots, \ell_n \rangle_{\Z}^{}; $$
then, $\order h_K (\Lambda) = (\Lambda : \Lambda\cap \Norm_{K/\Q} (K^\times)) 
= 2^n$, by computing norm residue symbols with the R\'edei matrix \eqref{array}
which is of rank $n$; whence $\CH_{K, 2} = \CH_{K, 2}^G = \CH_{K, 2}^\ram 
\simeq (\Z/2\Z)^{n-1}$. 
\end{proof}

A program computing these relations is given Appendix \ref{A2}.

\subsection{The non-trivial relation for \texorpdfstring{$n$}{Lg} even}
Consider the relations: 
\begin{equation}\label{eq1}
\left\{\begin{aligned}
{\mathfrak p}^{a_0} \prd_{i = 1}^n {\mathfrak l}_i^{a_i} & =:
{\mathfrak p}^{a_0} \prd_{i \in I} {\mathfrak l}_i  = (\alpha),\ \alpha \in K^\times, \\
\Norm_{K/\Q}(\alpha) & =  s\,2^{a_0} \prd_{i \in I} \ell_i, 
\ s \in \{\pm 1\}, 
\end{aligned}\right.
\end{equation}
where $I$ is a subset of $[1, n]$ and $s = \pm1$;
from the trivial relation $(\sqrt m) = 
\prod_{i=1}^n {\mathfrak l}_i$ and $\Norm_{K/\Q}(\sqrt m)
=-\prod_{i=1}^n \ell_i$ we deduce $(\alpha \sqrt m) = 
 (c)\,{\mathfrak p}^{a_0}\,\prod_{i=1}^n {\mathfrak l}_i^{\ov a_i}$, 
where $\ov a_i = 1-a_i$ and $c = \prod_{i \in I} \ell_i$; 
taking $\ov \alpha := \frac{\alpha \sqrt m}{c}$, one gets the 
equivalent complementary relations, where $\ov I = [1, n] \setminus I$:
\begin{equation}\label{eq2}
\left\{\begin{aligned}
{\mathfrak p}^{a_0} \prd_{i \in \ov I} {\mathfrak l}_i & = (\ov \alpha),\   \\
\Norm_{K/\Q}(\ov \alpha) &= -s\,2^{a_0}\prd_{i \in \ov I} \ell_i.
\end{aligned}\right.
\end{equation}

\begin{lemma}\label{lem2}
(i) In the relations \eqref{eq1} and \eqref{eq2}, one has $a_0=1$
and the two non-trivial equivalent relations may be written:

\smallskip
$\bullet$ ${\mathfrak p}\cdot  \prod_{j \in J} {\mathfrak l}_j = (\beta)$, 
$J \subseteq \{2, \ldots, n\}$, $\order J \equiv 1 \pmod 2$.

\smallskip
$\bullet$ ${\mathfrak p}\cdot  {\mathfrak l}_1 \prod_{j \in \ov J} {\mathfrak l}_j = (\ov \beta)$, 
$\ov J \subseteq \{2, \ldots, n\}$, $\order \ov J \equiv 0 \pmod 2$, 

\smallskip
The relation without ${\mathfrak l}_1$ is ${\mathfrak p}\cdot  {\mathfrak l}_2 = (\beta)$.

\medskip
(ii) These relations are given by the ideals $(\varepsilon+1)$ and $(\varepsilon-1)$ 
for which $\ffrac{\varepsilon+1}{\varepsilon-1} = C \cdot \sqrt m$, where $C$
is an odd rational number.

\medskip
(iii) There exists a sign $s_0$ such that $(\varepsilon+s_0) = 
(c)\,{\mathfrak p}\cdot  {\mathfrak l}_2$.
\end{lemma}

\begin{proof}
(i) Assume that $a_0 = 0$; let $\prod_{j \in J} {\mathfrak l}_j = (\beta)$
be the non-trivial relation such that $J \subseteq [2, n]$ (so, $J \ne \emptyset$).
Then $\prod_{j \in J} \ell_j = \pm \Norm_{K/\Q} (\beta)$; since
$\Norm_{K/\Q} (\beta) \equiv 1 \pmod 4$ as well as the $\ell_j$'s, we get
$\prod_{j \in J} \ell_j = \Norm_{K/\Q} (\beta)$.  
Consider the symbols $\Big(\hbox{$\ffrac{\prod_{j \in J} \ell_j }{\ell_i}$}\Big)$,
$i \in [1, n]$, using \eqref{array}:

\smallskip
$\bullet$ For $i=2$, one gets $\Big(\hbox{$\ffrac{\prod_{j \in J} \ell_j }{\ell_2}$}\Big)=1$
if and only if $\order J$ is even (which solves the case $n=2$ since $J = \{2\}$).

$\bullet$ for $i =1$, one obtains $\Big(\hbox{$\ffrac{\prod_{j \in J} \ell_j }{\ell_1}$}\Big)=1$
if and only if $j=2 \not\in J$; 

$\bullet$ for any $i \geq 3$, $\Big(\hbox{$\ffrac{\prod_{j \in J} \ell_j }{\ell_i}$}\Big)=1$
if and only if $j = i \not\in J$. 

\smallskip
So we obtain $J = \emptyset$ (absurd); whence the relation
${\mathfrak p}\cdot  \prod_{j \in J} {\mathfrak l}_j = (\beta)$
for $J \subseteq [2, n]$.

\smallskip
Then
$\Big(\hbox{$\ffrac{2\,\prod_{j \in J} \ell_j }{\ell_i}$}\Big) = 1$ for all $i \in [1, n]$ 
since $\Norm_{K/\Q} (\beta) = -2 \prod_{j \in J} \ell_j$ is not possible ($-1$ is not 
local norm at $2$); so this is equivalent to $\Big(\hbox{$\ffrac{\prod_{j \in J} \ell_j }
{\ell_i}$}\Big) = -1$ for all $i \in [1, n]$. 
The case $i = 2$ gives $\order J$ odd; the case $i \in J$, when $j=2 \not \in J$, 
gives a contradiction, as well as $2 \in J$ and $i \not\in J$, so that one verifies 
that only the relation ${\mathfrak p}\cdot  {\mathfrak l}_2 = (\beta)$ holds (the 
$(n+1) \times (n+1)$ matrix \eqref{array} is of rank $n$ and the sum of the two 
lines corresponding to $2$ and $\ell_2$ gives again the norm relation).

\smallskip
(ii) We have $(\varepsilon+s)^{1-\sigma} = s \varepsilon$; so $(\varepsilon+s)$ 
is a principal invariant ideal necessarily of the form $(c_s) {\mathfrak p}{\mathfrak m}_s$, 
$c_s \in \Q^\times$, ${\mathfrak p} \mid (2)$, ${\mathfrak m}_s \mid (\sqrt m)$.
Then, $\ffrac{\varepsilon+1}{\varepsilon-1} = C \cdot \sqrt m$ since 
$\Big(\ffrac{\varepsilon+1}{\varepsilon-1}\cdot \ffrac{1}{\sqrt m}\Big)^{1-\sigma} = 1$. 
\end{proof}

\begin{lemma}\label{lem3}
The case $b$ even does not occur.
\end{lemma}

\begin{proof}
Assume that $b$ is even; then $b = 4b'$; there exists a sign $s = \pm 1$
such that $a+s \equiv 2 \pmod 4$, so that $\varepsilon+s = a+s+4b' \sqrt m$
and the ideal  $(\varepsilon+s) = (a+s+4b' \sqrt m)$ is of the form 
$(2) (a''+2 b' \sqrt m)$ with $a''$ odd, which can not give the non-trivial relations
of Lemma \ref{lem2} since $a''+2 b' \sqrt m$ is ``odd'' (absurd).
\end{proof}

\begin{lemma}\label{lem4}
In the writing $\varepsilon = a+b \sqrt m$, $a = 4a'$ with $a'$ odd.
\end{lemma}

\begin{proof}
Assume that $a' = 2 a''$ and put $\theta_{s'} := \varepsilon+s' = 
8 a''+s'+b \sqrt m$, $s' \in \{\pm1\}$. We have, using $mb^2 = a^2-1$,
$\Norm_{K/\Q}(\theta_{s'}) = 2 (8 a''s'+1)$. 
From Lemma \ref{lem2}\,(iii) the non-trivial relation may be written
(where $c$ is an odd rational), 
$(\theta_{s'}) = (c)\,{\mathfrak p}\,{\mathfrak l}_2$ 
giving by taking the norm: 
$$8 a''s' +1 = s\,c^2\,{\mathfrak l}_2 \equiv 5s \pmod 8,\ \,s \in \{\pm 1\} ; $$
which is absurd.
\end{proof}

\subsection{Structure of \texorpdfstring{$\CT_{K, 2}$}{Lg} and triviality
of \texorpdfstring{$\CR_{K, 2}$}{Lg}}
Let's give the Kummer generators of  the maximal sub-extension, of
$H_{K, 2}^\pr/K$, of exponent $2$.

\begin{lemma}\label{lem5}
The module $\CT_{K, 2}$ is of $\Z_2$-rank $n-1$ and the maximal sub-extension 
of $H_{K, 2}^\pr/K_\infty$, of exponent $2$, is the Kummer extension
$K_\infty(\sqrt {\varepsilon}, \sqrt{\ell_1}, \ldots, \sqrt{\ell_n})$, under
the existence of the two independent relations $\sqrt\ell_1 \cdots \sqrt\ell_n 
= \sqrt m \in K^\times$ and $\sqrt 2 \sqrt{\ell_2} = 
\alpha \sqrt \varepsilon$, $\alpha \in K^\times$.
\end{lemma}

\begin{proof}
The maximal sub-extension of $H_{K, 2}^\pr/K$, of exponent $2$, is the Kummer 
extension $K(\sqrt{2}, \sqrt{\varepsilon}, \sqrt{\ell_1}, \ldots, \sqrt{\ell_n})$
(all the radicals $a$ in the $\sqrt a$'s are squares of ideals of $K$ and are totally positive; 
there is no other radicals since $\CH_{K, 2} = \CH_{K, 2}^\ram$).
We utilize the fact that $\sqrt 2 \in K_\infty$, then the relations
$\ell_1 \cdots \ell_n = m \in K^{\times 2}$ and 
${\mathfrak p}\,{\mathfrak l}_2 = (\alpha)$, $\alpha \in K^\times$, 
giving, by squaring, $2 \ell_2 = \alpha^2 \eta > 0$, $\eta \in E_K$, 
$\eta \ne 1$; then we may assume that $\eta = \varepsilon$, 
whence $2  \ell_2 = \alpha^2 \varepsilon$.

\smallskip
Whence the result since $\CT_{K, 2}$ is a direct factor in $\Gal(H_{K, 2}^\pr/K)$.
\end{proof}

\begin{theorem}\label{regulator}
The normalized $2$-adic regulator $\CR_{K, 2}$ is trivial. 
\end{theorem}

\begin{proof}
We have $\varepsilon = a+ b \sqrt m = 4 a'+b \sqrt m$ with $a^2-mb^2=1$; then:
$$\varepsilon^2 = a^2+2ab \sqrt m+mb^2 =1+2mb^2+8 a' b \sqrt m .$$

Recall that $a'$ is odd (Lemma \ref{lem4}), that $b$ is odd and  $m\equiv -1 \pmod 8$; thus 
$\varepsilon^2 \equiv 1-2 +8 \sqrt m \pmod {16}$; whence $\log_2(\varepsilon) \sim 4$.

\smallskip
From \cite[Proposition 5.2]{Gra8}, we have, in our context:
$$\order \CR_{K, 2} \sim \frac{1}{2} \cdot
\frac{\big(\Z_2 : \log(\Norm_{K/\Q}(\CU_{K, 2})) \big)}
{ \order \CW_{K, 2} \cdot \prod_{{\mathfrak p} \mid 2} \Norm {\mathfrak p}}
\cdot \frac {R_K}{\sqrt {D_K}}, $$

\noindent
where $R_K = \log_2(\varepsilon)$ is the usual $2$-adic regulator \cite[\S\,5.5]{W} and 
$D_K$ the discriminant of $K$. 
Since we have $\big(\Z_2 : \log(\Norm_{K/\Q}(\CU_{K, 2})) \big) = 
\big(\Z_2 : \log(1+4\,\Z_2) \big) = 4$, $R_K \sim 4$, 
$\order \CW_{K, 2} = 2$, $\prod_{{\mathfrak p} \mid 2}\Norm {\mathfrak p} = 2$, 
$\sqrt {D_K} \sim 2$, this yields $\order \CR_{K, 2} \sim 1$.

\smallskip
Formula \eqref{T=HRW} gives $\order \CT_{K, 2} = 2^{n-1} \times 1 \times 2 =2^n$, and
Corollary \ref{remfond} ends the proof of the Deng--Li conjecture $\delta = 3$.
\end{proof}

\section{\texorpdfstring{The \texorpdfstring{$p$}{Lg}-rank of $\BK_2(\BZ_K)$}{Lg} 
when \texorpdfstring{$p \geq 5$}{Lg}}\label{p>3}

We consider a prime number $p \geq 5$ and a cyclic $p$-extension
$K/k$ fulfilling some conditions:

\begin{theorem} \label{pdiv} Let $p$ be a prime $\geq 5$.
Let $k$ be a real abelian number field of prime-to-$p$ degree and 
let $K$ be any cyclic $p$-extension of $k$, abelian over $\Q$. 
We assume that $K \cap \Q(\mu_p) = \Q$ and that $p$
does not totally split in $K(\mu_p)/K$.
Let $S_\ta$ be the set of primes $\ell \ne p$, ramified in $K/k$, and
let $t_\ell \geq 1$ be the number of prime ideals above $\ell$ in $k/\Q$. Then 
$\rk_p \big(\BK_2(\BZ_K) \big) \geq \sum_{\ell \in S_\ta} t_\ell$.

\smallskip\noindent
In particular, for all primes $\ell \equiv 1 \pmod p$ and for any non-trivial $p$-extension 
$K/\Q$, contained in $\Q(\mu_\ell^{})/\Q$, we have $\BK_2(\BZ_K)[p^\infty] \ne 1$. 
\end{theorem}

\begin{proof}
Let $\zeta_p$ be a generator of $\mu_p^{}$ and let $\Q' := \Q(\zeta_p)$, 
$k' := k(\zeta_p)$ and $K' := K(\zeta_p)$; put $G := \Gal(K'/k')$ and $g := \Gal(K'/K)$. 
Recall that $\omega : g \to \mu_{p-1}^{}(\Z_p)$ is the Teichm\"uller character of $g$ 
(such that $\zeta_p^s =: \zeta_p^{\omega(s)}$, for all $s \in g$) and that any 
$\Z_p[g]$-module $X$ is the sum $\bigoplus_{j=1}^{p-1} X_{\omega^j}$ in an 
obvious meaning. Since $p \geq 5$, $\omega \ne \omega^{-1}$ and the 
comparison of $\BK_2(\BZ_K)[p^\infty] = \BK_2(\BR_K)[p^\infty]$ with 
$\CT_{K, p}$ made in \cite{GJ} does not hold.

\smallskip
The primes of $S_p(K)$ being not totally split in $K'/K$, the $\omega^{-1}$-component 
of the $\Z_p[G]$-module $\langle S_p(K') \rangle \otimes \Z_p$ generated by $S_p(K')$ 
above $S_p(K)$ is trivial; thus, formula \eqref{rankR} becomes:
$$ \rk_p(\BK_2(\BZ_K)) = \rk_p\big((\CH_{K', p})_{\omega^{-1}} \big). $$

We note that a prime number $\ell$, ramified in $K/\Q$ with ramification 
index divisible by $p$, fulfills the condition $\ell \equiv 1 \pmod p$
(for generalizations to the non-Galois case, see \cite[II, \S\,(d)]{Gr5}); whence 
$\ell$ totally splits in $\Q'/\Q$ and, for ${\mathfrak l} \mid \ell$ in $k$, 
the $p-1$ prime ideals ${\mathfrak L}$ above ${\mathfrak l}$ in $k'$ ramify 
in $K'/k'$ with the same ramification index; this gives $(p-1) t_\ell$ ramified 
ideals in $K'/k'$, whence $\order S_\ta(k') = (p-1)\sm_{\ell \in S_\ta} t_\ell$. 

Let $G_v \subseteq G$ be the inertia group in $K'/k'$ of a ramified place 
$v$ of $k'$ (so $v$ is a prime ideal above $\ell \ne p$ or a prime ideal above 
$p$, but the $\omega^{-1}$-component of $\langle S_p(k') \rangle \otimes \Z_p$ 
is trivial), and put $S(k') = S_\ta(k')\,\cup\,S_p(k')$.
Let $E_{k'}$ be the group of units of $k'$:

\smallskip
The Chevalley--Herbrand formula in $K'/k'$ writes:
\begin{equation*}
\order \CH_{K', p}^G = \order \CH_{k', p}^\nr \times \frac{\order \Omega_{k'}}
{\order h_{k'}(E_{k'})}, 
\end{equation*}
where $\CH_{k', p}^\nr \subseteq \CH_{k', p}$ corresponds to 
$\Gal(H_{k'}^\nr/K' \cap H_{k'}^\nr)$ and 
$\Omega_{k'} =\Big \{(s_v)_v \in \bigoplus_{v \in S(k')} 
G_v,\ \ \,\prod_v s_v = 1 \Big\}$, 
$h_{k'} : E_{k'} \to \Omega_{k'}$ being the map defined by the 
family of Hasse norm residue symbols $\Big(\big(\ffrac{x, K'/k'}{v}\big)\Big)_v$
which fulfill the ``product formula'' on units. 
This formula is associated to the following exact sequences: 
\begin{equation} \label{suites}
\left\{\begin{aligned}
&1 \to \BJ_{K'/k'}(\CH_{k', p}) \cdot \CH_{K', p}^\ram \too \CH_{K', p}^G \\
&\hspace{1.7cm} \too  E_{k'} \cap \Norm_{K'/k'}(K'^\times)/\Norm_{K'/k'}(E_{K'}) \to 1, \\
&1 \to E_{k'} / E_{k'} \cap \Norm_{K'/k'}(K'^\times) \\
&\hspace{1.7cm}\too \Omega_{k'} \too \Gal(H_{K'/k'}^\gen/K' H_{k'}^\nr) \to 1, 
\end{aligned}\right.
\end{equation}
where $\BJ_{K'/k'}$ is the extension of classes, $\CH_{K', p}^\ram$ the 
subgroup of $\CH_{K', p}$ generated by the ramified primes, $H_{K'/k'}^\gen 
\subseteq H_{K'}^\nr$ the genus field, fixed by $\CH_{K', p}^{1-\sigma}$, 
where $\sigma$ generates $G$ \cite[Proposition IV.4.5]{Gr5}.

\smallskip
The link between the two aspects (fixed points and genus 
exact sequences) is given by $\Gal(H_{K'/k'}^\gen/K') \simeq 
\CH_{K', p}/ \CH_{K', p}^{1-\sigma}$ and the exact sequence
$1 \to  \CH_{K', p}^G \to  \CH_{K', p} \to \CH_{K', p}^{1-\sigma} \to 1$. 

\smallskip
More precisely, we need the $\omega^{-1}$-component of $\CH_{K', p}^G$
and of the terms of the exact sequences. For $\CH_{K', p}^G$, this computation 
has been done in Jaulent's Thesis \cite[Chapitre III]{Jau1} and reproduced in 
\cite[\S\,2.2, Theorem 2.1\,(ii)]{Gra10}; the difficulty comes from the fact that 
$\CH_{k', p}$ is not necessary isomorphic to a sub-module of $\CH_{K', p}$ 
(indeed, only $\BJ_{K'/k'}(\CH_{k', p})$ makes sense as sub-module of 
$\CH_{K', p}$ and $\BJ_{K'/k'}$ is not necessarily injective). 
It writes under our context:
\begin{equation}\label{omega}
\order (\CH_{K', p}^G)_{\omega^{-1}} = \order (\CH_{k', p}^\nr)_{\omega^{-1}} \times 
\frac{\order (\Omega_{k'})_{\omega^{-1}}}{\order (h_{k'}(E_{k'}))_{\omega^{-1}}} .
\end{equation}

Since $E_{k'} = E_{k'^+} \oplus \mu_p$ (up to a prime-to-$p$ index), where 
the subgroup $E_{k'^+}$ of real units of $k'$ is of even character, and since 
$\mu_p$ is of character $\omega \ne \omega^{-1}$, the $\omega^{-1}$-component 
of $h_{k'}(E_{k'})$ is trivial and exact sequences \eqref{suites} and formula \eqref{omega} 
reduce to:
\begin{equation}\label{isomorphisms}
\left\{\begin{aligned}
(\CH_{K', p}^G)_{\omega^{-1}} & \simeq \big (\BJ_{K'/k'}(\CH_{k', p}) \cdot 
\CH_{K', p}^\ram \big)_{\omega^{-1}}, \\
(\Omega_{k'})_{\omega^{-1}} & \simeq (\Gal(H_{K'/k'}^\gen/K' H_{k'}^\nr))_{\omega^{-1}}, \\
\order (\CH_{K', p}^G)_{\omega^{-1}} & = \order (\CH_{k', p}^\nr)_{\omega^{-1}} \times 
\order (\Omega_{k'})_{\omega^{-1}} . \\
\end{aligned}\right.
\end{equation}

Whence:
\begin{equation}\label{fin}
\left\{\begin{aligned}
\rk_p \big(\BK_2(\BZ_K)\big) & = \rk_p\big((\CH_{K', p})_{\omega^{-1}} \big) \\
& \geq \rk_p(\CH_{K', p} / \CH_{K', p}^{1-\sigma})_{\omega^{-1}}
\geq \rk_p(\Omega_{k'})_{\omega^{-1}}. 
\end{aligned}\right.
\end{equation}

Let $G' := \langle G_v \rangle_v \simeq \Gal(K'/K' \cap H_{k'}^\nr)$; the product 
formula may be interpreted by means of the exact sequence:
$$1 \to \Omega_{k'} \too \plus_{v \in S(k')} G_v \too G' \subseteq G \to 1, $$
in which the character of $G'$ and that of $\bigoplus_{v \in S_p(k')} G_v$
is the unit character. Whence $(\Omega_{k'})_{\omega^{-1}} \simeq
\plus_{v \in S_\ta(k')} G_v$.

For $\ell \in S_\ta$, let $p^{e_\ell}$, $e_\ell \geq 1$, be the ramification index 
of $\ell$ in $K'/k'$. The inertia groups $G_v$, for $v \mid \ell$
in $k'$, does not depend on $v$ and by abuse may be denoted $G_\ell$. 
Since each of the $t_\ell$ primes ${\mathfrak l} \mid \ell$ in $k$, totally splits in 
$k'/k$ into $p-1$ primes ${\mathfrak L} \mid {\mathfrak l}$ of $k'$, one may write:
$$\plus_{{\mathfrak L}} G_{\mathfrak L} = \plus_{\ell \in S_\ta} \
\plus_{{\mathfrak l} \mid \ell} \
\plus_{{\mathfrak L} \mid {\mathfrak l}} G_\ell \simeq
\plus_{\ell \in S_\ta} \Big(\big (\Z/p^{e_\ell} \Z \big)^{p-1} \Big)^{t_\ell}, $$ 
in which, each $\ell$-component is, as Galois module, isomorphic 
to $t_\ell$ copies of the regular representation 
$\bigoplus_{j=1}^{p-1}\big (\Z[G_\ell]/
p^{e_\ell} \Z[G_\ell] \big)_{\omega^j}$ 
whose $\omega^{-1}$-component is of order $p^{e_\ell}$. Whence
$(\Omega_{k'})_{\omega^{-1}} \!\simeq \bigoplus_{\ell \in S_\ta} 
(\Z/p^{e_\ell} \Z)^{t_\ell}$. 

Thus, one obtains, from \eqref{fin}, $\rk_p(\BK_2(\BZ_K)) \geq \sum_{\ell \in S_\ta} t_\ell$. 
\end{proof}

\begin{corollary}
We obtain the following consequences:

\smallskip
(i) $\order (\CH_{k', p}^\nr)_{\omega^{-1}} \times 
\order (\Omega_{k'})_{\omega^{-1}} = \order \Big (\BJ_{K'/k'}(\CH_{k', p}) \cdot 
\CH_{K', p}^\ram \Big)_{\omega^{-1}}$. 

\smallskip
(ii) If $\BK_2(\BZ_k)[p^\infty] = 1$, then $(\CH_{K', p}^\ram)_{\omega^{-1}} 
\simeq \bigoplus_{\ell \in S_\ta} (\Z/p^{e_\ell} \Z)^{t_\ell}$.

\smallskip
(iii) If $k = \Q$, then:

\smallskip
\qquad $\bullet$ $\rk_p(\BK_2(\BZ_K)) \geq  \order S_\ta$; 

\smallskip
\qquad $\bullet$ $(\CH_{K', p}^\ram)_{\omega^{-1}} \simeq 
\bigoplus_{\ell \in S_\ta} \Z/p^{e_\ell} \Z$.
\end{corollary}

\begin{proof}
Cases (i), (ii) come from the fact that $\BK_2(\BZ_k)[p^\infty] = 1$ is equivalent to 
$(\CH_{k', p})_{\omega^{-1}}  = 1$, then to \eqref{isomorphisms}. If $k = \Q$, 
$\sum_{\ell \in S_\ta} t_\ell = \order S_\ta$; then $(\CH_{\Q', p})_{\omega^{-1}} = 1$
from Ribet's reciprocal of Herbrand's criterion, since the Bernoulli number 
$B_2 = \frac{1}{6}$ is prime to $p$ \cite[Theorem 5, p.\,45]{Rib}. Whence (iii).
\end{proof}

\begin{remarks}
(i) For $p=3$, $\omega^{-1} = \omega$, so that $(\CH_{K', p}^G)_{\omega^{-1}}$
depends on $\ds \frac{\order (\Omega_{\Q'})_{\omega^{-1}}}{\order h_{\Q'}(\mu_3)}$ 
which may be trivial, when $S_\ta = \{\ell\}$ and $[K : \Q] = 3$.

\smallskip
(ii) Replacing $\BK_2(\BZ_K)[p^\infty]$ by $\CT_{K, p}$ (whatever the prime $p \geq 3$) 
and $\omega^{-1}$ by $\omega$, $h_{\Q'}(\mu_p)$ may be non-trivial, being of 
character $\omega$. So, if $K/\Q$ is of prime conductor $\ell$, $\CH_{K, p} = 1$ 
and we may have $\CT_{K, p}=1$ but $\BK_2(\BZ_K)[p^\infty] \ne 1$.

\smallskip
(iii) For a table of numerical examples, see Appendix \ref{D1} for the 
computation of $\order\BK_2(\BZ_K)[p^\infty]$, $p \geq 3$, and 
Appendix \ref{D2} for the computation of $\rk_5(\BK_2(\BZ_K)[5^\infty]$.
Numerical tables are given in \cite{Br2} for cyclic cubic fields; more on 
computations of $p$-ranks are given in Qin's papers \cite{Q1, Q2, Q3}.
\end{remarks}

\section{Conclusion}
These explicit calculations lead to a confirmation of the properties of the 
$\BL_p$-functions when $s$ varies in $\Z_p$ and takes the values $s \in \{-1, 1\}$, 
giving annihilation theorems, orders of isotypic components of classical $p$-invariants. 

\smallskip
Other invariants than class groups, torsion groups and regular kernels exist, as for 
instance the Jaulent logarithmic class group $\wt \CH_{K, p}$\,\footnote{\,Also denoted, 
most often in the literature, $\wt {\rm C\ell}_{K,p}$, $\wt \CT_{K,p}$ as quotient of $\CT_{K,p}$;
but it looks more like the $p$-class group $\CH_{K, p}$ and there exists a formula of the 
form $\order \CT_{K,p} = \wt \CH_{K, p} \cdot \wt \CR_{K, p} \cdot \wt \CW_{K, p}$, 
for a suitable ``logarithmic regulator'', see \cite[Schema \S\,2.3]{Jau24}. Moreover 
it may capitulates in real $p$-extensions of $K$ (the key for Greenberg's conjecture),
while $\CT_{K,p}$ never capitulates under Leopoldt's conjecture.} \cite{Jau23, Jau3, Jau4} 
linked to the wild Hilbert kernel \cite[Th\'eor\`eme 9]{Jau22} and in relation with Greenberg's 
conjecture \cite{Jau24, Jau25} (see Appendices \ref{clog2}, \ref{clog3} for some computations), 
and invariants of the higher $K$-theory (giving conjectural expressions suggested 
in \cite[Section 12, Conjecture 12.2]{Gr4} and some results of \cite{JM}). 

\smallskip
An important problem is to find the structure of these invariants, it being understood 
that they fulfill, once genus part has been taken into consideration, standard densities 
in the spirit of the Cohen--Lenstra--Martinet--Malle distributions. A specific technique, 
of genus theory type, does exist in cyclic $p$-extensions, especially for $p$-class 
groups with determining the canonical filtration by means of a natural algorithm 
generalizing Chevalley--Herbrand fixed points formula (large bibliography, 
synthesized in \cite{Gr7}). 

\smallskip
A main question being to find such algorithms 
for the other $p$-invariants, which unfortunately does not exist to our knowledge;
the case of the most important groups $\CT_{K, p}$ being very interesting; indeed, 
class field theory (by means of computation of suitable ray class groups) 
gives the group structure, but this computational aspect does not give a method
to determine the filtration in cyclic $p$-extensions; only the fixed points formula, 
is known \cite[Theo\-rem IV.3.3]{Gr5}.

\smallskip
In a cohomological viewpoint, $\CT_{K, p}$ is closely related to the 
Tate--Chafarevich group
$\Cha_{K, p}^2 := {\rm Ker} \big [{\Hom}^2 (\CG_{K, S}, \F_p) \rightarrow
\oplus_{{\mathfrak p} \mid p}\,{\Hom}^2 (\CG_{K_{\mathfrak p}}, \F_p) \big]$, 
and that $\CT_{K, p} \simeq \Hom^2(\CG_{K, S}, \Z_p)^\ast$ \cite[Th\'eor\`eme1.1]{Ng1}, 
where $\CG_{K, S}$ is the Galois group of the maximal $p$-ramified pro-$p$-extension 
of $K$ and $\CG_{K_{\mathfrak p}}$ the local analogue over $K_{\mathfrak p}$; 
generalizations of the same kind are done in the literature, as in 
\cite[Theorem 0.4, 0.7, 0.11,\,\ldots]{Ko1}, \cite[Theorem 3.5, Example 4.3]{AAM}, 
with analogous cohomological framework, without any workable algorithm.

\smallskip
A major difficulty is to interpret, in terms of ideals, global units and so on, such
cohomology groups; the case of $\CT_{K, p} \simeq \Hom^2(\CG_{K, S}, \Z_p)^\ast$
is edifying since $\CG_{K, S}$ is inaccessible, while $\CT_{K, p}$ depends on the 
$p$-class group and on the normalized $p$-adic regulator, of $K$. Of course, cohomology
gives important and non-obvious relations but, in the previous example, it is the knowledge 
of $\CT_{K, p}$ which yields important information on $\CG_{K, S}$ (e.g., case of the simplest
structure giving the notion of $p$-rational fields issued from \cite{Gr3, Mov} and largely 
developed in a lot of papers; see \cite[Appendix]{Gra88} for a survey 
about this notion and general abelian $p$-ramification theory).
  
\smallskip
Maybe the link we made in \cite[\S\,III.2, Theorem III.2.6, Theorem III.3.3]{Gr5} and
\cite[Appendice]{GJ}, using logarithms of ideals for an effectiveness of the arithmetic, 
may be generalized.

\smallskip
As a final remark, we point out the fact that these analytic genus theory aspects of $p$-adic 
pseudo-measures are stated in \cite{Gr33} for any totally real base field $k$ replacing $\Q$ 
(under Leopoldt's conjecture for $p$, to get the important result of deployment \cite[Theorem 
III.4.1.5]{Gr5}), and allowing similar technics with genus fields frome Deligne--Ribet 
pseudo-measures; these results being completed with for instance the article by Maire 
\cite{Mai} on the class field theory aspects and governing fields about the {\it existence} 
of prescribed $p$-cyclic ramified extensions for which an analog of Theorem \ref{thmfond} 
should apply.

\subsection*{Acknowledgements} I thank with pleasure Litong Deng and Yong\-xiong Li
for interesting exchanges about their paper, giving to me the idea of taking again into 
account my old papers on genus theory of $p$-adic pseudo-measures, hoping that 
arithmetic invariants of new families of real abelian fields can be elucidated via this method.

\begin{appendix}

\section{Data relating to the Deng--Li family}

This section is only concerned by the family of quadratic fields described 
and studied in Section \ref{DengLi}; we illustrate and confirm various numerical aspects.

\subsection{Computation of \texorpdfstring{$\CH_{K, p}$}{Lg}, 
\texorpdfstring{$\CT_{K, p}$}{Lg} and \texorpdfstring{$\wt \CH_{K, p}$}{Lg}}\label{A1}

We compute the three fundamental invariants for this family, it being understood that
the structure of $\CH_{K, p}$ and complex analytic properties are given in \cite{DL}
in order to study $\BK_2(\BZ_K)[2^\infty]$.

\subsubsection{General program for \texorpdfstring{$\CH_{K, p}$}{Lg}, 
\texorpdfstring{$\CT_{K, p}$}{Lg}}
It may be convenient to recall the program computing the structures of $\CH_{K, p}$ and $\CT_{K, p}$ 
for a number field $K$, for any ${\sf p \in [bp, Bp]}$, as soon as a defining polynomial $P \in \Z[X]$ 
(monic, irreducible) is given (in ${\sf P}$); for the computation of $\CT_{K, p}$, the parameter 
${\sf nu}$ must be chosen such that ${\sf p^{nu}}$ be larger than the exponent of the result 
${\sf T}$ \cite[\S\,2.1]{Gr6}; the number ${\sf r}$ is the number $r_2+1$ 
of independent $\Z_p$-extensions of $K$:

\smallskip
\ft\begin{verbatim}
{P=x^3-7*x+1;bp=2;Bp=5*10^5;K=bnfinit(P,1);r=K.sign[2]+1;
print("P=",P," H=",K.cyc);forprime(p=bp,Bp,nu=6;
Kpn=bnrinit(K,p^nu);HKn=Kpn.cyc;T=List;e=matsize(HKn)[2];
R=0;for(k=1,e-r,c=HKn[e-k+1];v=valuation(c,p);if(v>0,R=R+1;
listinsert(T,p^v,1)));if(R>0,print("p=",p," rk(T)=",R," T=",T)))}

P=x^3-7*x+1     H=[]        P=x^8-8*x+1    H=[4]
p=7    rk(T)=1  T=[7]       p=2   rk(T)=1  T=[4]
p=701  rk(T)=1  T=[701]     p=3   rk(T)=1  T=[3]
\end{verbatim}\ns

\subsubsection{Case of the Deng--Li family}
The following program verifies the structure of $\CH_{K, 2}$ and 
$\CT_{K, 2}$ for the Deng--Li family of quadratic fields, for $n=4$:

\smallskip
\ft\begin{verbatim}
{L=List
([7215,26455,77415,119535,142935,153735,166335,171015,196359,
226655,241215,243295,257335,283855,311415,315055,420135,430495,
447135,473415,475215,490295,504295,545415,550615,552695,553335,
563695,568815,592215,603655,606615,633399,657735,665223,673215,
685815,687895,727935,751335,755495,757055,790495,798135,803751,
807455,818935,833199,849615,878415,884455,886015,886335,896415,
905255,911495,934935,961935,973655,981695,990015]);
p=2;nu=9;for(j=1,61,P=x^2-L[j];K=bnfinit(P,1);HK=K.cyc;
Kpn=bnrinit(K,p^nu);HKn=Kpn.cyc;T=List;e=matsize(HKn)[2];
for(k=1,e-1,c=HKn[e-k+1];v=valuation(c,p);if(v>0,listput(T,p^v)));
print("P=",P," T=",T," H=",HK))}

P=x^2-7215  T=[2,2,4] H=[2,2,2]   P=x^2-981695 T=[2,2,4] H=[10,2,2]
P=x^2-26455 T=[2,2,4] H=[2,2,2]   P=x^2-990015 T=[2,2,4] H=[18,2,2]
\end{verbatim}\ns

\smallskip
From \eqref{equalranks}, the $2$-ranks of $\BR_2(\BZ_K)$ and $\CT_{K, 2}$ are equal to $3$, 
this implies, since $\BK_2(\BZ_K)[2^\infty] \simeq (\Z/2\Z)^3 \times \Z/2^3\Z$, that:
$$\BR_2(\BZ_K)[2^\infty] \simeq (\Z/2\Z)^2 \times \Z/2^2\Z \simeq \CT_{K, 2}. $$ 

More generally, for the Deng--Li family, we have:
$$\BR_2(\BZ_K)[2^\infty] \simeq \CT_{K, 2} \simeq 
(\Z/2\Z)^{n-2} \!\times \Z/2^2\Z,\  \CH_{K, 2} \simeq (\Z/2\Z)^{n-1}. $$ 

From this, we verify that $\CW_{K, 2} \simeq \Z/2\Z$, which
is coherent with the formula $\order \CT_{K, 2} = \order \CH_{K, 2} \times 
\order \CR_{K, 2} \times \order \CW_{K, 2}$ since the normalized $2$-adic regulator 
is trivial from Theorem \ref{regulator}.

\subsubsection{Logarithmic class group}\label{classlog}
 
The computation of $\wt \CH_{K, 2}$ (instruction ${\sf CLog=bnflog(K, 2)}$) 
for the Deng--Li family, gives always the structure $ (\Z/2\Z)^2$ for $n=4$; 
but we do not know if this invariant is given by the $\BL_2(s, \chi)$'s at some 
$s \in \Z_2$ (see \S\,\ref{clog2} for the computation of $\wt \CH_{K, 2}$ for  
quadratic fields and \S\,\ref{clog3} for the cubic case):

\smallskip
\ft\begin{verbatim}
P=x^2-7215     CLog=[[2,2],[],[2,2]]  T=[2,2,4] H=[2,2,2]
P=x^2-26455    CLog=[[2,2],[],[2,2]]  T=[2,2,4] H=[2,2,2]
(...)
P=x^2-981695   CLog=[[2,2],[],[2,2]]  T=[2,2,4] H=[10,2,2]
P=x^2-990015   CLog=[[2,2],[],[2,2]]  T=[2,2,4] H=[18,2,2]
\end{verbatim}\ns

\smallskip\noindent
Moreover, the third component of ${\sf bnflog(K, 2)}$ gives the $S_2$-class group
$\CH_{K, 2}/{\rm cl}({\mathfrak p}) \simeq \Z/2\Z \times \Z/2\Z$, where 
${\mathfrak p} \mid 2$ is ramified; this result confirms that necessarily 
${\mathfrak p}$ is non-principal of order $2$.

\subsection{The canonical non-trivial relation}\label{A2}
The program gives the non-trivial relation whose writing does not contain ${\mathfrak l}_1$
(with $\ell_1 \equiv 3 \pmod 8$), then that obtained multiplying by $(\sqrt m)$.
 We have proven that this relation is always
${\mathfrak p} {\mathfrak l}_2 = (\alpha)$ (Lemma \ref{lem2}).
The list $L$ gives the exponents $a_0$, $a_1$, $a_2,\,\ldots$, 
of ideals ${\mathfrak p}$, ${\mathfrak l}_1$, ${\mathfrak l}_2,\,\ldots$, in this order 
especially for the particular primes ${\mathfrak l}_1$, ${\mathfrak l}_2$ (which 
is not so immediate since the PARI factorization of $m$ gives primes 
in the ascending order). The generator $\alpha$ of the relation of principality is  
$\alpha = u+v \sqrt m$; for checking, its norm $N$ and the sign 
$S$ of $N$ are computed:

\smallskip
\ft\begin{verbatim}
{m=List
([7215,26455,77415,119535,142935,153735,166335,171015,196359,
226655,241215,243295,257335,283855,311415,315055,420135,430495,
447135,473415,475215,490295,504295,545415,550615,552695,553335,
563695,568815,592215,603655,606615,633399,657735,665223,673215,
685815,687895,727935,751335,755495,757055,790495,798135,803751,
807455,818935,833199,849615,878415,884455,886015,886335,896415,
905255,911495,934935,961935,973655,981695,990015]);p=2;
for(j=1,61,M=m[j];D=quaddisc(M);r=omega(D);L0=List;
for(i=1,r,listput(L0,0));P=x^2-M;K=bnfinit(P,1);print();Lel=List;
Div=component(factor(D),1);listput(Lel,2,1);h=3;for(i=2,r,c=Div[i];
if(Mod(c,4)==-1,el1=c;listput(Lel,c,2)));for(i=2,r,c=Div[i];
if(c!=el1 & kronecker(c,el1)==-1,el2=c;listput(Lel,c,3)));
for(i=2,r,c=Div[i];if(c!=el2 & Mod(c,4)==1,h=h+1;listput(Lel,c,h)));
print("M=",M," D=",Lel);for(k=1,2^r-1,B=binary(k);t=#B;
LB=L0;for(i=1,t,listput(LB,B[i],r-t+i));F=idealfactor(K,D);
Fel=List;CFel=component(F,1);listput(Fel,CFel[1],1);h=3;
for(i=2,r,c=CFel[i][1];if(Mod(c,4)==-1,el1=c;
listput(Fel,CFel[i],2)));
for(i=2,r,c=CFel[i][1];if(c!=el1 & kronecker(c,el1)==-1,el2=c;
listput(Fel,CFel[i],3)));
for(i=2,r,c=CFel[i][1];if(c!=el2 & Mod(c,4)==1,h=h+1;
listput(Fel,CFel[i],h)));
A=1;for(j=1,r,e=LB[j];if(e==0,next);A=idealmul(K,A,Fel[j]);
Q=bnfisprincipal(K,A);u=Q[2][1];v=Q[2][2];
if(Q[1]==0 & u!=0,print("L=",LB);N=u^2-M*v^2;S=sign(N);
print(component(factor(abs(N)),1));
print("u=",u," v=",v," S=",sign(S)))))}
\end{verbatim}\ns

\smallskip
Let's give some excerpts showing all possible cases:

\smallskip
\ft\begin{verbatim}
M=7215 D=[2,3,5,13,37]                M=26455 D=[2,11,13,5,37]
L=[1,0,1,0,0]                         L=[1,0,1,0,0]
[2,5]                                 [2,13]
u=-85 v=1  S=1                        u=86654841 v=-532769  S=1
L=[1,1,0,1,1]                         L=[1,1,0,1,1]
[2,3,13,37]                           [2,5,11,37]
u=1443 v=17  S=-1                     u=1084184915 v=-6665757  S=-1

M=504295 D=[2,11,173,5,53]            M=665223 D=[2,3,461,13,37]
L=[1,0,1,0,0]                         L=[1,0,1,0,0]
[2,173]                               [2,461]
u=7970980451749 v=11224562253 S=1     u=10603 v=13  S=1
L=[1,1,0,1,1]                         L=[1,1,0,1,1]
[-1,2,5,11,53]                        [2,3,13,37]
u=-32719598967495 v=46075031513 S=-1  u=18759 v=23  S=-1
\end{verbatim}\ns

\section{Quadratic fields}

\subsection{Computation of 
\texorpdfstring{$\order \BK_2(\BZ_K)[2^\infty]$, $\CT_{K, 2}$ and $\BC$}{Lg}} \label{B1}
The following program computes, instead, $\order \BR_2(\BZ_K)[2^\infty]$ for any real 
quadratic field (so $f$ is $m$ or $4m$) and $\CT_{K, 2}$ for 
checking of Theorem \ref{thmfond} (cases (ii${}_2$) and (iii${}_2$) and
property {\bf (b)} about equality v.s. inequality, for all $s \in \Z_p$).
In that cases, $\epsilon = 1$, $\BD$ is the number of odd ramified primes, 
and $\BC = \BD$ (resp. $\BC = \BD-1$) if $\theta^{-1}\chi(2) = 1$
(resp. $\theta^{-1}\chi(2) \ne 1$). 

\smallskip
Recall that ${\sf nu}$ and ${\sf n}$ must be chosen large enough (below we take 
$n=10$ for $m >100$, $n=13$ for $m > 10^3$).

\smallskip
We do not write the trivial cases $\CT_{K, 2} = \BR_2(\BZ_K)[2^\infty] = 1$.

\smallskip
\ft\begin{verbatim}
{p=2;q=4;nu=10;n=6;for(m=5,2000,if(core(m)!=m,next);f=quaddisc(m);
P=x^2-m;K=bnfinit(P,1);Kpn=bnrinit(K,p^nu);H=Kpn.cyc;T=List;
e=matsize(H)[2];valT=0;for(k=1,e-1,h=H[e-k+1];v=valuation(h,p);
if(v>0,valT=valT+v;listput(T,p^v)));f=quaddisc(m);fn=q*p^n*f;c=1;
while(gcd(c,2*m)!=1 || kronecker(f,c)!=-1,c=c+1);S=0;
forstep(a=1,fn/2,2,if(gcd(a,m)!=1,next);aa=lift(Mod(a/c,fn));
la=(aa*c-a)/fn;eps=kronecker(m,a);S=S+a*eps*(la+(1-c)/2));
valR=valuation(S,2)-1;D=omega(m);if(Mod(m,2)==0,D=D-1);
C=D;if(Mod(m,8)!=-1,C=D-1);
if(valT==C,print("m=",m," c=",c," T=",T," v_2(T)=",valT,
" val(R_2Z)=",valR," C=",C," Equality"));
if(valT>C,print("m=",m," c=",c," T=",T," v_2(T)=",valT,
" v_2(R_2Z)=",valR," C=",C," Inequality")))}

m=7   c=5  T=[4]     v_2(T)=2  v_2(R_2Z)=2  C=1 Inequality
m=14  c=3  T=[2]     v_2(T)=1  v_2(R_2Z)=1  C=0 Inequality
m=15  c=13 T=[4]     v_2(T)=2  v_2(R_2Z)=2  C=2 Equality
m=17  c=3  T=[2]     v_2(T)=1  v_2(R_2Z)=1  C=0 Inequality
m=21  c=11 T=[2]     v_2(T)=1  v_2(R_2Z)=1  C=1 Equality
m=23  c=3  T=[4]     v_2(T)=2  v_2(R_2Z)=2  C=1 Inequality
m=30  c=11 T=[2]     v_2(T)=1  v_2(R_2Z)=1  C=1 Equality
m=31  c=7  T=[8]     v_2(T)=3  v_2(R_2Z)=3  C=1 Inequality
m=33  c=5  T=[2]     v_2(T)=1  v_2(R_2Z)=1  C=1 Equality
m=34  c=7  T=[2]     v_2(T)=1  v_2(R_2Z)=1  C=0 Inequality
m=35  c=3  T=[2]     v_2(T)=1  v_2(R_2Z)=1  C=1 Equality
m=39  c=11 T=[4]     v_2(T)=2  v_2(R_2Z)=2  C=2 Equality
m=41  c=3  T=[16]    v_2(T)=4  v_2(R_2Z)=3  C=0 Inequality
(...)
m=1001 c=3  T=[2,2]  v_2(T)=2  v_2(R_2Z)=2  C=2 Equality
m=1002 c=5  T=[16]   v_2(T)=4  v_2(R_2Z)=5  C=1 Inequality
m=1003 c=5  T=[4]    v_2(T)=2  v_2(R_2Z)=2  C=1 Inequality
m=1005 c=29 T=[2,8]  v_2(T)=4  v_2(R_2Z)=5  C=2 Inequality
m=1006 c=7  T=[2]    v_2(T)=1  v_2(R_2Z)=1  C=0 Inequality
m=1007 c=3  T=[4]    v_2(T)=2  v_2(R_2Z)=2  C=2 Equality
m=1009 c=11 T=[2]    v_2(T)=1  v_2(R_2Z)=1  C=0 Inequality
m=1010 c=3  T=[2]    v_2(T)=1  v_2(R_2Z)=1  C=1 Equality
m=1011 c=7  T=[4]    v_2(T)=2  v_2(R_2Z)=2  C=1 Inequality
m=1015 c=17 T=[2,4]  v_2(T)=3  v_2(R_2Z)=3  C=3 Equality
m=1022 c=3  T=[2,16] v_2(T)=5  v_2(R_2Z)=9  C=1 Inequality
m=1023 c=5  T=[2,64] v_2(T)=7  v_2(R_2Z)=7  C=3 Inequality
\end{verbatim}\ns

\smallskip
The equality $v_2(\order \CT_{K, 2}) = \BC$ often occur and when
$v_2(\order \CT_{K, 2}) > \BC$, then $v_2(\order \BR_2(\BZ_K)[2^\infty] > \BC$, 
with some cases of distinct valuations ($m=41, 65, 66, 114,\,\ldots$).

\smallskip
Numerical examples given in \cite[Table]{DL}, yield equalities with 
$\BC = \BD +v_2(4 s) -2 \epsilon = 4$ for $s=-1$, since $2$ splits in 
$M = \Q(\sqrt{-m})$ (Theorem \ref{thmfond}\,(iii$_2$) and Corollary \ref{remfond}).

\subsection{Computation of \texorpdfstring{$\wt \CH_{K, 2}$}{Lg}}\label{clog2}
 
We give a program of computation of the logarithmic class group to compare with 
the invariants $\CH_{K, 2}$, $\CT_{K, 2}$. There are many structures, but the 
relation between $\CH_{K, 2}$ and $\CT_{K, 2}$ being well-known, we give only 
an excerpt of interesting examples concerning $\wt \CH_{K, 2}$. For $p=2$, the 
first component of ${\sf bnflog(K,2)}$ gives $\wt \CH_{K, 2}^0$ of index 
$1$ or $2$ in $\wt \CH_{K, 2}$ \cite[\S\,3, Remark 3.1]{BeJa}:

\smallskip
\ft\begin{verbatim}
{bf=10^4;Bf=10^4+10^3;for(m=bf,Bf,if(core(m)!=m,next);PK=x^2-m;
F=component(factor(m),1);nu=12;K=bnfinit(PK,1);HK=K.clgp;
Kpn=bnrinit(K,2^nu);HKn=Kpn.cyc;T=List;w=0;e=matsize(HKn)[2];
for(k=1,e-1,CK=HKn[e-k+1];v=valuation(CK,2);if(v>0,w=w+v;
listput(T,2^v)));K=bnfinit(PK,1);HK=K.clgp;Kpn=bnrinit(K,2^nu);
HKn=Kpn.cyc;T=List;w=0;e=matsize(HKn)[2];for(k=1,e-1,CK=HKn[k+1];
v=valuation(CK,2);if(v>0,w=w+v;listput(T,2^v)));CLK=bnflog(K,2);
print("m=",F," P=",PK," H=",HK[2]," T=",T," CLog=",CLK))}

m=[3,5,23,29] x^2-10005 H=[2,2]  T=[4,2,2] CLog=[[2,2],[],[2,2]]
m=[3,47,71]   x^2-10011 H=[4,2]  T=[4,2]   CLog=[[4],[],[4]]
m=[3,5,11,61] x^2-10065 H=[2,2]  T=[16,2,2]CLog=[[4],[2],[2]]
m=[2,3,23,73] x^2-10074 H=[8,2]  T=[32,4]  CLog=[[8],[],[8]]
m=[2,3,19,89] x^2-10146 H=[4,2]  T=[8,2]   CLog=[[4],[],[4,2]]
m=[7,31,47]   x^2-10199 H=[4,2]  T=[4,4,2] CLog=[[4,2],[],[4,2]]
m=[10337]     x^2-10337 H=[]     T=[8]     CLog=[[4],[4],[]]
m=[2,5297]    x^2-10594 H=[8]    T=[32]    CLog=[[4],[],[8]]
m=[3,11,17,19]x^2-10659 H=[2,2,2]T=[2,2,2] CLog=[[2,2,2],[],[2,2,2]]
m=[17,641]    x^2-10897 H=[2]    T=[64,2]  CLog=[[2,2],[2],[2]]
m=[3,5,17,47] x^2-11985 H=[2,2]  T=[1024,2,2] 
                                           CLog=[[128],[64],[2]]
m=[12161]     x^2-12161 H=[11]   T=[512]   CLog=[[16],[16],[]]
m=[2,3,13,167]x^2-13026 H=[2,2]  T=[128,2] CLog=[[],[],[2]]
\end{verbatim}\ns

\subsection{Computation of \texorpdfstring{$\CT_{K, 3}$ and 
$\order \BK_2(\BZ_K)[3^\infty]$}{Lg}} \label{B2}

For $p=3$, genus theory does not apply; we only give the valuation 
of $\order \BK_2(\BZ_K)[3^\infty]$ and the structure of $\CT_{K, 3}$; trivial 
cases are not written:

\smallskip
\ft\begin{verbatim}
{p=3;nu=5;n=6;for(m=2,10^2,if(core(m)!=m,next);P=x^2-m;
K=bnfinit(P,1);Kpn=bnrinit(K,p^nu);H=Kpn.cyc;T=List;e=matsize(H)[2];
for(k=1,e-1,C=H[e-k+1];v=valuation(C,p);if(v>0,listput(T,p^v)));
f=quaddisc(m);fn=p^(n+1)*f;c=1;while(Mod(c,p)==0 || 
kronecker(f,c)!=-1,c=c+1);S=0;forstep(a=1,fn/2,[1,2],
if(gcd(a,f)!=1,next);aa=lift(Mod(a/c,fn));la=(aa*c-a)/fn;
eps=kronecker(f,a);S=S+a*eps*(la+(1-c)/2));print
("m=",m,"    v_3(K_2Z)=",valuation(S,3)," T=",T))}

m=6      v_3(K_2Z)=1  T=[3]      m=69     v_3(K_2Z)=1  T=[3]
m=15     v_3(K_2Z)=1  T=[3]      m=74     v_3(K_2Z)=1  T=[9]
m=29     v_3(K_2Z)=1  T=[9]      m=77     v_3(K_2Z)=1  T=[3]
m=33     v_3(K_2Z)=1  T=[3]      m=78     v_3(K_2Z)=1  T=[3]
m=42     v_3(K_2Z)=3  T=[9]      m=79     v_3(K_2Z)=1  T=[9]
m=43     v_3(K_2Z)=2  T=[3]      m=82     v_3(K_2Z)=4  T=[3]
m=51     v_3(K_2Z)=1  T=[3]      m=83     v_3(K_2Z)=1  T=[3]
(...)
m=10187  v_3(K_2Z)=2  T=[81]     m=10673  v_3(K_2Z)=4  T=[27]
m=10239  v_3(K_2Z)=2  T=[3,9]    m=10771  v_3(K_2Z)=6  T=[9]
m=10281  v_3(K_2Z)=1  T=[27]     m=10842  v_3(K_2Z)=4  T=[3,3]
m=10297  v_3(K_2Z)=4  T=[3]      m=10942  v_3(K_2Z)=2  T=[3,243]
m=10351  v_3(K_2Z)=1  T=[243]    m=11062  v_3(K_2Z)=3  T=[27]
\end{verbatim}\ns

\section{Cyclic cubic fields}

\subsection{Computation of \texorpdfstring{$\order \BK_2(\BZ_K)[3^\infty]$, }{Lg}
\texorpdfstring{$\CT_{K, 3}$}{Lg} and \texorpdfstring{$\BC$}{Lg}}\label{C1}
We consider all cyclic cubic fields, which requires 
calculating the Artin group giving $\Gal(\Q(\mu_f^{})/K)$; indeed, a 
composite conductor $f$ gives rise to several fields, so that we must work from 
the defining polynomial depending on suitable integers $a$, $b$ such that 
$f = \frac{a^2+27\,b^2}{4}$ (another method would be to work in the group 
$(\Z/f\Z)^\times$, giving other difficulties). We compare 
$v_{\mathfrak m} \big(\order \BK_2(\BZ_K)[3^\infty] \big)$ to  
$\BC$ to see if we are, or not, in case of equality $v_{\mathfrak m} 
\big(\BL_3(s, \chi) \big) = \BC,\,\forall\,s \in \Z_3$.

\subsubsection{General program}
The program gives at first the list of all cyclic cubic fields, of 
conductor $f$ (in ${\sf f}$, ${\sf f \in [bf, Bf]}$), with a defining 
polynomial $P_K$ (in ${\sf PK}$). 
We compute for $p=3$ the structures of the $3$-class group (in ${\sf H}$) 
that of the $3$-torsion group (in ${\sf T}$) and $\order \BK_2(\BZ_K)[3^\infty]$
(in ${\sf \order K{\_{}\!\_{}} 2 Z}$). 

\smallskip
Except if the first layer of the
cyclotomic $\Z_3$-extension of $K$ is non-ramified (e.g., ${\sf f=657,\ 
P=x^3-219*x+730,\ H_K=[3, 3],\ T=[3]}$), then $\CT_{K, 3}=1$ implies 
$\CH_{K, 3} = \CR_{K, 3} = 1$ since $\CW_{K, 3}=1$:

\smallskip
\ft\begin{verbatim}
{bf=7;Bf=10^3;Q=y^2+y+1;Y=Mod(y,Q);for(f=bf,Bf,
if(isprime(f)==1,next);h=valuation(f,3);if(h!=0 & h!=2,next);
F=f/3^h;if(core(F)!=F,next);F=factor(F);Div=component(F,1);
d=matsize(F)[1];for(j=1,d,D=Div[j];if(Mod(D,3)!=1,break));
Df=component(factor(f),1);for(b=1,sqrt(4*f/27),
if(h==2 & Mod(b,3)==0,next);A=4*f-27*b^2;if(issquare(A,&a)==1,
if(h==0,if(Mod(a,3)==1,a=-a);PK=x^3+x^2+(1-f)/3*x+(f*(a-3)+1)/27);
if(h==2,if(Mod(a,9)==3,a=-a);PK=x^3-f/3*x-f*a/27);
nu=8;K=bnfinit(PK,1);HK=K.clgp;Kpn=bnrinit(K,3^nu);HKn=Kpn.cyc;
T=List;w=0;e=matsize(HKn)[2];for(k=1,e-1,CK=HKn[e-k+1];
v=valuation(CK,3);if(v>0,w=w+v;listput(T,3^v)));
C=omega(f)-1;if(Mod(f,3)==0,C=C-1);forprime(q=2,10^2*f,
if(Mod(3*f,q)!=0 & polisirreducible(PK+O(q))==1,q0=q;break));
A0=List;A1=List;A2=List;c=q0;n=3;fn=3^n*f;N=eulerphi(fn)/(2*3);
for(s=1,fn/2,if(gcd(s,f*3)!=1,next);for(k=0,fn,r=s+k*fn;
if(isprime(r)==0,next);if(polisirreducible(PK+O(r))==0,
listput(A0,s);break)));S0=0;S1=0;S2=0;g=(1-c)/2;
for(j=1,N,s=A0[j];listput(A1,lift(Mod(s*q0,fn))));
for(j=1,N,s=A1[j];listput(A2,lift(Mod(s*q0,fn))));
for(j=1,N,s=A0[j];t=lift(Mod(s/c,fn));ls=(t*c-s)/fn;S0=S0+s*(ls+g));
for(j=1,N,s=A1[j];t=lift(Mod(s/c,fn));ls=(t*c-s)/fn;S1=S1+s*(ls+g));
for(j=1,N,s=A2[j];t=lift(Mod(s/c,fn));ls=(t*c-s)/fn;S2=S2+s*(ls+g));
S=S0+Y*S1+Y^2*S2;K2=norm(S)/3;v=valuation(K2,3);
if(v==C,print("f=",Df," P=",PK," H=",HK[2],
" T=",T," #K_2Z=",3^v," C(s)=",C," Equality"));
if(v>C,print("f=",Df," P=",PK," H=",HK[2],
" T=",T," #K_2Z=",3^v," C(s)=",C," Inequality")))))}
\end{verbatim}\ns

\subsubsection{Numerical results} We only write an excerpt of the results:

\smallskip
\ft\begin{verbatim}
f=[9,7]  P=x^3-21*x-35       H=[3] T=[]   #K_2Z=1  C(s)=0 Equality
f=[9,7]  P=x^3-21*x+28       H=[3] T=[]   #K_2Z=1  C(s)=0 Equality        
f=[7,13] P=x^3+x^2-30*x-64   H=[3] T=[3]  #K_2Z=3  C(s)=1 Equality
f=[7,13] P=x^3+x^2-30*x+27   H=[3] T=[3]  #K_2Z=3  C(s)=1 Equality        
f=[9,19] P=x^3-57*x-152      H=[3] T=[3]  #K_2Z=3  C(s)=0 Inequality
f=[9,19] P=x^3-57*x+19       H=[3] T=[3]  #K_2Z=3  C(s)=0 Inequality
f=[7,31] P=x^3+x^2-72*x+209  H=[3] T=[3,3]#K_2Z=9  C(s)=1 Inequality
f=[7,31] P=x^3+x^2-72*x-225  H=[3] T=[3]  #K_2Z=3  C(s)=1 Equality
f=[9,37] P=x^3-111*x+370     H=[3] T=[3]  #K_2Z=3  C(s)=0 Inequality
f=[9,37] P=x^3-111*x+37      H=[3] T=[3]  #K_2Z=3  C(s)=0 Inequality
f=[13,31]P=x^3+x^2-134*x-597 H=[3] T=[3]  #K_2Z=3  C(s)=1 Equality
f=[13,31]P=x^3+x^2-134*x+209 H=[3] T=[3]  #K_2Z=3  C(s)=1 Equality
f=[7,61] P=x^3+x^2-142*x+601 H=[3] T=[3]  #K_2Z=3  C(s)=1 Equality
f=[7,61] P=x^3+x^2-142*x-680 H=[3] T=[9,9]#K_2Z=9  C(s)=1 Inequality
f=[7,67] P=x^3+x^2-156*x-799 H=[3] T=[3,3]#K_2Z=27 C(s)=1 Inequality
f=[7,67] P=x^3+x^2-156*x+608 H=[3] T=[3]  #K_2Z=3  C(s)=1 Equality
(...)
f=[9,73] P=x^3-219*x-1241 H=[3,3] T=[3] #K_2Z=3 C(s)=0 Inequality
\end{verbatim}\ns

\subsection{Computation of \texorpdfstring{$\wt \CH_{K, 3}$}{Lg}}\label{clog3}
 
We use a part of the previous program, adding the computation of the logarithmic 
class group and we write some data obtained up to conductors $\leq 10^3$, then 
for larger conductors. The first component of ${\sf bnflog(K, 3)}$ gives 
$\wt \CH_{K, 3}$, the second one (often trivial) is the sub-group 
$\wt {\rm cl} (\langle S_3\rangle)$ generated by $S_3$ and the third component 
is $\CH_{K, 3}^{S_3} = \CH_{K, 3}/{\rm cl} (\langle S_3\rangle)$; we have 
$\wt \CH_{K, 3}/\wt {\rm cl} (\langle S_3\rangle) \simeq \CH_{K, 3}^{S_3}$ \cite{BeJa}. 

\smallskip
\ft\begin{verbatim}
f=[19]     x^3+x^2-6*x-7      H=[]    T=[3]   CLog=[[],[],[]]
f=[7,13]   x^3+x^2-30*x-64    H=[3]   T=[3]   CLog=[[3],[],[3]]
f=[9,61]   x^3-183*x-854      H=[3]   T=[]    CLog=[[],[],[3]]
f=[19,37]  x^3+x^2-234*x-729  H=[2,6] T=[3,3] CLog=[[3],[3],[]]
f=[7,181]  x^3+x^2-422*x+3191 H=[3,3] T=[3,3] CLog=[[3,3],[],[3,3]]
(...)
f=[7,19,73,103] P=x^3+x^2-333342*x+73964960 
        H=[3,3,3,3] T=[3,3,3,3]     CLog=[[3,3,3,3],[],[3,3,3,3]]
f=[7,19,73,103] P=x^3+x^2-333342*x-69038901 
        H=[6,6,3]   T=[3,3,3,3]     CLog=[[3,3,3],[3],[3,3]]
f=[7,19,73,103] P=x^3+x^2-333342*x-49038361 
        H=[3,3,3]   T=[3,3,3]       CLog=[[3,3,3],[],[3,3,3]]
f=[397,251893] P=x^3+x^2-33333840*x+41296924413 
        H=[3]       T=[9,3,3]       CLog=[[3,3],[3,3],[]]
f=[7,14286007] P=x^3+x^2-33334016*x-56482638787 
        H=[54,18]   T=[27,27]       CLog=[[27,9],[],[27,9]]
f=[7,163,87643] P=x^3+x^2-33333554*x-214816239 
        H=[3,3]     T=List([3,3,3]) CLog=[[3,3],[3],[3]]
f=[100001053] P=x^3+x^2-33333684*x-17437220649 
        H=[]        T=List([9,3])   CLog=[[9,3],[9,3],[]]
f=[31,103,31321] P=x^3+x^2-33335984*x+68190607953 
        H=[6,6,3]   T=List([9,9,3]) CLog=[[9,3],[3],[3,3]]
\end{verbatim}\ns

\subsection{Computation of \texorpdfstring{$\order \BK_2(\BZ_K)[2^\infty]$
and $\CT_{K, 2}$}{Lg}}\label{C2}

We still compute $\order \BR_2 (\BZ_K) [2^\infty] \sim \frac{1}{2} \cdot \BL_2(-1, \chi) 
\times \frac{1}{2} \cdot \BL_2(-1, \chi^2)$, for prime conductors, noting that genus 
theory does not apply for cyclic cubic fields with $p=2$. We compute $\CT_{K, 2}$, 
at first with the classical program using ray class groups, then we use formula of 
Theorem \ref{Lp} in the following complete program (we only print an example 
of the six cases of structures that occur up to $f=3500$):

\smallskip
\ft\begin{verbatim}
{p=2;Q=y^2+y+1;Y=Mod(y,Q);forprime(f=7,5*10^5,
if(Mod(f,3)!=1,next);for(b=1,sqrt(4*f/27),A=4*f-27*b^2;
if(issquare(A,&a)==1,if(Mod(a,3)==1,a=-a);
P=x^3+x^2+(1-f)/3*x+(f*(a-3)+1)/27;K=bnfinit(P,1);nu=8;
Kpn=bnrinit(K,p^nu);HKn=Kpn.cyc;T=List;e=matsize(HKn)[2];
for(k=1,e-1,CK=HKn[e-k+1];v=valuation(CK,p);if(v>0,listput(T,p^v)));
g=znprimroot(f);c=lift(g);if(Mod(c,p)==0,c=c+f);n=8;ex=(f-1)/3;
cex=Mod(c,f)^ex;R0=0;R1=0;R2=0;T0=0;T1=0;T2=0;fn=p^(n+2)*f;
forstep(a=1,fn/2,2,if(Mod(a,f)==0,next);aa=lift(Mod(a/c,fn));
am=lift(Mod(a,fn)^-1);la=(aa*c-a)/fn;XR=a*(la+(1-c)/2);XT=am*la;
aex=Mod(a,f)^ex;if(aex==1,R0=R0+XR;T0=T0+XT);if(aex==cex,
R1=R1+XR;T1=T1+XT);if(aex==cex^2,R2=R2+XR;T2=T2+XT));
AR=lift(Mod(R0-R2,fn));BR=lift(Mod(R1-R2,fn));
AT=lift(Mod(T0-T2,fn));BT=lift(Mod(T1-T2,fn));
NR=norm(AR+Y*BR);wR=valuation(NR,p);
NT=norm(AT+Y*BT);wT=valuation(NT,p);print();print("f=",f," P=",P);
print("AR=",AR," BR=",BR,"  #R_2Z=",p^wR);print("T=",T);
print("AT=",AT," BT=",BT,"  #T=",p^wT))))}

f=7 P=x^3+x^2-2*x-1            f=739 P=x^3+x^2-246*x-520
  AR=114683 BR=6       #R_2Z=1   AR=7232 BR=4072        #R_2Z=64
  T=[]                           T=[8,8]
  AT=45707 BT=57070       #T=1   AT=2395776 BT=8911992     #T=64
  
f=31 P=x^3+x^2-10*x-8          f=2689 P=x^3+x^2-896*x+5876
  AR=24 BR=507870      #R_2Z=4   AR=1733440 BR=461096   #R_2Z=64
  T=[2,2]                        T=[16,16]
  AT=82976 BT=73118       #T=4   AT=29791312 BT=37866256  #T=256
  
f=277 P=x^3+x^2-92*x+236       f=3163 P=x^3+x^2-1054*x-13472
  AR=3256 BR=4537692  #R_2Z=16   AR=432 BR=51769200    #R_2Z=256
  T=[4,4]                        T=[16,16]
  AT=718480 BT=1119412   #T=16   AT=11946416 BT=40520656  #T=256
  
f=3457 P=x^3+x^2-1152*x+13700  f=6163 P=x^3+x^2-2054*x+17576
AR=286536 BR=3460600  #R_2Z=64   AR=6140160 BR=6195120 #R_2Z=256
T=List([16,16])                  T=List([8,8])
AT=2216784 BT=1012224   #T=256   AT=2754032 BT=3116408     #T=64
\end{verbatim}\ns

\smallskip
We know from \eqref{equalranks} that $\BR_2 (\BZ_K) [2^\infty]$ and $\CT_{K, 2}$
have same $2$-rank, but we see that most often the whole structures coincide 
(first exceptions $f = 2689, 3457$ and $f=6163$ giving an interesting case).

\subsection{Zeroes of \texorpdfstring{$p$}{Lg}-adic \texorpdfstring{$\BL$}{Lg}-functions 
and \texorpdfstring{$\BC(s)$}{Lg}}\label{C3}

A method, to study the influence of the zeroes, close to $s=1$, of $p$-adic 
$\BL$-functions, is to construct families of degree-$p$ cyclic fields $K_n$ such that 
the $p$-adic regulator tends to $0$ as $n \to \infty$, so $\BL_p(1, \chi_n) \to 0$ 
and $\order \CT_{K_n, p} \to 0$ from \eqref{T=HRW}; thus, invariants given by 
some $\BL_p(s_0^{}, \chi_n)$, $s_0^{} \in \Z_p^\times$, have the same behavior 
because of the role of $\BC(s) = \BC$ when $s \in \Z_p^\times$ (Theorem\ref{thmfond}).

\subsubsection{The Lecacheux--Washington cyclic cubic fields}
The following family of cyclic cubic fields $K$ illustrates, for $p=3$, the fact that, 
in a genus theory context, the strict inequality (Theorem \ref{thmfond}\,{\bf (b)}):
$$v_{\mathfrak m} \big(\ffrac{1}{2} \BL_3(s, \chi) \big) > \BC,\,\forall\,s \in \Z_3, $$ 
can be largely exceeded because of $\BL_3(1, \chi)$ very close to $0$:

\smallskip
Using the family of cubic polynomials \cite{Lec, W4}:
$$P_n=x^3-(n^3-2n^2+3n-3) x^2-n^2 x-1, $$ 
for some values of $n \to 1$ in $\Z_3$, one obtains $3$-adic $\BL$-functions having 
a zero near $s=1$, which of course gives large modules $\CT_{K_n, 3}$. 
The required conditions are given in \cite[Theorem 3]{W4}.

\subsubsection{Numerical examples}
The constant $\BC$ associated to these fields is always $2$, so that we must obtain 
$v_3(\CT_{K, 3}) > 2$ as soon as $n$ is close to $1$ in $\Z_3$.
We give the program and some examples suggesting that $\order \CT_{K_n, 3}$ is 
unbounded in such families, as $n \to \infty$; but these fields have an huge discriminant 
$D_n$ and the function $\mathcal C_p(K_n) := \ds \frac{\log(\order \CT_{K_n, p})}
{\log(\sqrt {D_n})}$, that we have introduced in \cite{Gra9}, is rather small and fulfills 
the various conjectures given in this article:

\smallskip
\ft\begin{verbatim}
{M=3^5;forstep(n=1,10^6,M,if(isprime(n^2+3)==0,next);
A=n^2-3*n+3;if(numdiv(A)!=4,next);F=factor(A);
p=component(F,1)[1];q=component(F,1)[2];
if(Mod(p^2,9)==1 || Mod(q^2,9)==1,next);
P=x^3-(n^3-2*n^2+3*n-3)*x^2-n^2*x-1;K=bnfinit(P,1);
nu=18;HK=K.clgp;Kpn=bnrinit(K,3^nu);HKn=Kpn.cyc;
T=List;w=0;e=matsize(HKn)[2];for(k=1,e-1,CK=HKn[e-k+1];
v=valuation(CK,3);if(v>0,w=w+v;listput(T,3^v)));
print("n=",n," p=",p," q=",q," T=",T,"    v_3(T)=",w))}

n=2674   p=43     q=166099     T=[3,81,243]       v_3(T)=10
n=40096  p=3271   q=491461     T=[3,729,729]      v_3(T)=13
n=43498  p=6547   q=288979     T=[3,81,243]       v_3(T)=10
n=50788  p=28111  q=91753      T=[3,243,243]      v_3(T)=11
n=56134  p=3613   q=872089     T=[3,729,2187]     v_3(T)=14
n=76546  p=31     q=189001951  T=[3,2187,2187]    v_3(T)=15
n=78490  p=7      q=880063519  T=[3,81,243]       v_3(T)=10
n=96958  p=79     q=118994467  T=[3,9,243,729]    v_3(T)=14
n=124660 p=7      q=2219963089 T=[3,6561,6561]    v_3(T)=17
n=158194 p=135301 q=184957     T=[3,243,729]      v_3(T)=12
(...)
n=570808 p=643    q=506718601  T=[3,6561,19683]   v_3(T)=18
n=649540 p=229    q=1842359227 T=[3,3,19683,59049]v_3(T)=21
\end{verbatim}\ns

\noindent
Note that $n=570808 \equiv 1 \! \pmod {3^9}$ and $n=649540 \equiv 1\! \pmod {3^{10}}$.

\medskip
For $M = 3$, only giving $n \equiv 1 \pmod 3$, we obtain many equalities:

\smallskip
\ft\begin{verbatim}
n=340  p=7     q=16369   T=[3,3]     v_3(T)=2
n=736  p=79    q=6829    T=[3,3]     v_3(T)=2
n=970  p=7     q=133999  T=[3,3]     v_3(T)=2
(...)
\end{verbatim}\ns

\smallskip
For $M = 9$, giving $n \equiv 1 \pmod 9$, we never obtain
a valuation $2$, but minimal valuations are $4$:

\smallskip
\ft\begin{verbatim}
n=874  p=7     q=108751  T=[3,3,9]   v_3(T)=4
n=2926 p=1777  q=4813    T=[3,3,9]   v_3(T)=4
n=4042 p=3067  q=5323    T=[3,3,9]   v_3(T)=4
(...)
\end{verbatim}\ns

More generally,  the minimal valuations are $2\,v_3(M)$.

\section{Numerical results for \texorpdfstring{$p \geq 5$}{Lg}}

When $p \geq 5$ and for a cyclic $p$-extension $K/\Q$, the trick used
for $p \in \{2, 3\}$ does not exist and we only have available the complex
analytic formulas of $\order \BK_2(\BZ_K)$ and the rank formula
\eqref{rankR} giving Theorem \ref{pdiv}.

\subsection{Computation of \texorpdfstring{$\order \BK_2(\BZ_K)[p^\infty]$}{Lg}}
\label{D1}

The computation of the order of $\BK_2(\BZ_K)$ uses Bernoulli 
polynomials \cite[Theorem 4.2 using Proposition 4.1]{W}.

\smallskip
The following program computes the $p$-valuation of the product
$\prod_{\chi \ne \1} \BL(-1, \chi)$ for the $p-1$ characters
$\chi$ of order $p$ and prime conductor$\ell$ with the formula
$\BL(-1, \chi) \sim \sum_{a=1}^{\ell-1} \chi(a)\,(a^2-\ell a)$,   
where $\chi(a)$ is computed from $\chi(g) := \zeta_p$, where $g$ is a 
primitive root modulo $\ell$ and from the writing $a \equiv g^k \pmod \ell$
for $k \in [1, \ell-1]$; by comparison we compute the structure of $\CT_{K, p}$.
The case $p=3$ is included to recall that it is particular and may give trivial
modules; moreover, for $p>3$, equality \eqref{equalranks} does not apply
in general:

\smallskip
\ft\begin{verbatim}
{forprime(p=3,100,Q=polcyclo(p);X=Mod(x,Q);print();
forprime(ell=1,10^4,if(Mod(ell,p)!=1,next);g=znprimroot(ell);
L=0;for(k=1,ell-1,a=lift(g^k);E=a^2-a*ell;L=L+E*X^k);
w=valuation(norm(L),p);P=polsubcyclo(ell,p);K=bnfinit(P,1);
nu=8;Kpn=bnrinit(K,p^nu);HKn=Kpn.cyc;T=List;e=matsize(HKn)[2];
for(k=1,e-1,c=HKn[e-k+1];v=valuation(c,p);if(v>0,listput(T,p^v)));
print("p=",p," ell=",ell," T=",T," v(K_2Z)=",w)))}

p=3 ell=7     T=[]     v(K_2Z)=0    p=11 ell=23   T=[]     v(K_2Z)=1
p=3 ell=19    T=[3]    v(K_2Z)=1    p=11 ell=727  T=[11]   v(K_2Z)=1
p=3 ell=199   T=[3,3]  v(K_2Z)=2    p=11 ell=1321 T=[]     v(K_2Z)=3
p=3 ell=4177  T=[9,9]  v(K_2Z)=3    p=11 ell=1453 T=[11]   v(K_2Z)=1
p=3 ell=2593  T=[3,9]  v(K_2Z)=6    p=11 ell=3631 T=[11]   v(K_2Z)=2
p=3 ell=21997 T=[27,81]v(K_2Z)=7    p=11 ell=4357 T=[11,11]v(K_2Z)=1
p=5 ell=11    T=[]     v(K_2Z)=1    p=13 ell=53   T=[]     v(K_2Z)=1
p=5 ell=101   T=[5,5]  v(K_2Z)=1    p=13 ell=677  T=[13]   v(K_2Z)=1
p=5 ell=181   T=[]     v(K_2Z)=3    p=13 ell=1483 T=[]     v(K_2Z)=2
p=5 ell=401   T=[5,5]  v(K_2Z)=2    p=13 ell=2029 T=[13]   v(K_2Z)=1
p=5 ell=3001  T=[5]    v(K_2Z)=7    p=13 ell=6761 T=[13]   v(K_2Z)=4
p=5 ell=5351  T=[5,5,5]v(K_2Z)=1    p=13 ell=11831T=[13]   v(K_2Z)=2
p=7 ell=29    T=[]     v(K_2Z)=1    p=17 ell=103  T=[]     v(K_2Z)=2
p=7 ell=127   T=[]     v(K_2Z)=2    p=17 ell=137  T=[]     v(K_2Z)=1
p=7 ell=197   T=[7,7]  v(K_2Z)=1    p=17 ell=3469 T=[17]   v(K_2Z)=1
p=7 ell=491   T=[7]    v(K_2Z)=1    p=17 ell=3571 T=[]     v(K_2Z)=1
p=7 ell=4159  T=[]     v(K_2Z)=3    p=17 ell=3673 T=[]     v(K_2Z)=1
p=7 ell=4229  T=[]     v(K_2Z)=4    p=17 ell=3911 T=[]     v(K_2Z)=1
\end{verbatim}\ns

\smallskip
For $p$ large, the $p$-valuation is almost often $1$ with a trivial $\CT_{K, p}$.

\subsection{Computations of \texorpdfstring{$\rk_5 (\BK_2(\BZ_K))$}{Lg}}
\label{D2}

The next programs compute the $\omega^{-1}$-component of the $5$-class 
group $\CH_{K_2}$ of $K_2 := K(\mu_5^{})$, for degree-$5$ cyclic 
fields $K$, of prime conductor in the first section, then with $f=11 \cdot 31$ 
in the second section, thus giving the $5$-rank 
of $\BK_2\BZ_K[5^\infty]$  (formula \eqref{rankR}, where $\langle S_5(K') \rangle$
is of character $\chi_0^{}$).

\smallskip
The program computes the $5$-class group $\CH_{K_1}$ of 
$K_1 := K(\sqrt 5$), so that
$\CH_{K_2} = \CH_{K_1} \plus\  (\CH_{K_2})_{\omega^{-1}} \plus\ 
(\CH_{K_2})_{\omega}$,  
where $\rk_5 (\CH_{K_2})_{\omega} = \rk_5(\CT_{K, 5})$ (formula \eqref{rankT}), 
where the structure of $\CT_{K, 5}$ is computed as usual via the program
of Section \ref{A1}. Whence:
$$\rk_5 (\BK_2(\BZ_K)) = \rk_5 (\CH_{K_2})-\rk_5 (\CH_{K_1}) 
- \rk_5 (\CT_{K, 5}). $$

\subsubsection{Case of prime conductors}
So, $f$ is a prime number $\ell \equiv 1 \pmod 5$; thus, $\CH_K = 1$.
We write an excerpt of each structures obtained:

\smallskip
\ft\begin{verbatim}
{p=5;forprime(ell=11,10^5,if(Mod(ell,p)!=1,next);
P=polsubcyclo(ell,p);print("conductor ell=",ell," P=",P);nu=8;
K=bnfinit(P,1);HK=K.clgp[2];Kpn=bnrinit(K,p^nu);HKn=Kpn.cyc;
TK=List;w=0;e=matsize(HKn)[2];for(k=1,e-1,Ck=HKn[e-k+1];
v=valuation(Ck,p);if(v>0,listput(TK,p^v)));
R1=polcompositum(P,polsubcyclo(p,2))[1];LK1=bnfinit(R1,1);
R2=polcompositum(P,polcyclo(p))[1];LK2=bnfinit(R2,1);
print("TK=",TK," HK1=",LK1.clgp[2]," HK2=",LK2.clgp[2])}

conductor ell=11 P=x^5+x^4-4*x^3-3*x^2+3*x+1
    TK=[]     HK1=[]   HK2=[5]
conductor ell=31 P=x^5+x^4-12*x^3-21*x^2+x+5
    TK=[]     HK1=[]   HK2=[10,10,2,2]
conductor ell=101 P=x^5+x^4-40*x^3+93*x^2-21*x-17
    TK=[5,5]  HK1=[]   HK2=[2005,5,5]
conductor ell=151 P=x^5+x^4-60*x^3-12*x^2+784*x+128
    TK=[5]    HK1=[]   HK2=[5305,5]
conductor ell=251 P=x^5+x^4-100*x^3-20*x^2+1504*x+1024
    TK=[5]    HK1=[]   HK2=[11810,10,2,2]
conductor ell=281 P=x^5+x^4-112*x^3-191*x^2+2257*x+967
    TK=[]     HK1=[5]  HK2=[261005,5,5]
conductor ell=401 P=x^5+x^4-160*x^3+369*x^2+879*x-29
    TK=[5,5]  HK1=[]   HK2=[42505,5,5,5]
conductor ell=421 P=x^5+x^4-168*x^3+219*x^2+3853*x-3517
    TK=[]     HK1=[5]  HK2=[225915,15,3,3]
\end{verbatim}\ns

\smallskip
The rank $2$ is obtained for $\ell = 31, 281, 401$.

\subsubsection{An example of composite conductor}
The next program computes the $\omega^{-1}$-component  
for the four degree-$5$ cyclic fields $K$ of conductor $f=11 \cdot 31$ 
(for which $\CT_{K, 5} = 1$). One may replace ${\sf f=11*31}$ by any 
product of primes $\ell_1 \equiv \ell_2 \equiv 1 \pmod 5$):

\smallskip
\ft\begin{verbatim}
{p=5;f=11*31;LP=polsubcyclo(f,p);d=matsize(LP)[2];
for(i=1,d,P=LP[i];D=nfdisc(P);if(D!=f^(p-1),next);
print();print(P);K=bnfinit(P,1);HK=K.clgp[2];
nu=8;Kpn=bnrinit(K,p^nu);HKn=Kpn.cyc;
TK=List;w=0;e=matsize(HKn)[2];for(k=1,e-1,Ck=HKn[e-k+1];
v=valuation(Ck,p);if(v>0,listput(TK,p^v)));
R1=polcompositum(P,polsubcyclo(p,2))[1];LK1=bnfinit(R1,1);
R2=polcompositum(P,polcyclo(p))[1];LK2=bnfinit(R2,1);
print("HK=",HK," TK=",TK," HK1=",LK1.clgp[2]," HK2=",LK2.clgp[2])}

x^5-x^4-136*x^3+641*x^2-371*x-67
    HK=[5]  TK=[]  HK1=[5,5]  HK2=[218305,5,5,5,5]
x^5-x^4-136*x^3-723*x^2-1053*x-67
    HK=[5]  TK=[]  HK1=[5,5]  HK2=[145505,5,5,5,5]
x^5-x^4-136*x^3-41*x^2+3039*x-1431
    HK=[5]  TK=[]  HK1=[5,5]  HK2=[3805,5,5,5,5,5,5]
x^5-x^4-136*x^3+300*x^2+2016*x-3136
    HK=[5]  TK=[]  HK1=[5,5]  HK2=[12605,5,5,5,5,5]
\end{verbatim}\ns

\smallskip
The third case gives $\rk_5((\BK_2\BZ_K[5^\infty]_{\omega^{-1}}) = 4$
and the last one gives the rank $3$.
With $f=11 \cdot 101$, one obtains interesting structures, giving 
the ranks $2, 2, 4, 3$, respectively:

\smallskip
\ft\begin{verbatim}
x^5-x^4-444*x^3+3644*x^2+80*x-32000
    HK=[5]  TK=[5]  HK1=[5,5]  HK2=[1196410,10,10,10,5]
x^5-x^4-444*x^3+311*x^2+32299*x+26883
    HK=[5]  TK=[5]  HK1=[5,5]  HK2=[15020305,5,5,5,5]
x^5-x^4-444*x^3-1911*x^2+12301*x-8669
    HK=[5]  TK=[5]  HK1=[5,5]  HK2=[1405,1405,5,5,5,5,5]
x^5-x^4-444*x^3+311*x^2+43409*x+4663
    HK=[5]  TK=[5]  HK1=[5,5]  HK2=[5559605,5,5,5,5,5]
\end{verbatim}\ns

\smallskip
These ranks may indeed be larger than the number of tame ramified primes; 
this illustrates Theorem \ref{pdiv}.

\subsection{Computation of \texorpdfstring{$\BL(-1-(p-3)\,p^n, \chi)$}{Lg}}\label{D3}

Let $K/\Q$ be a degree-$p$ cyclic extension. For $p \geq 5$ the conditions 
\eqref{flp} ($m>1$ and $m \equiv 0 \pmod{p-1}$) are not 
fulfilled for $m=2$, which does not allow to express $\BL(-1, \chi)$ by means of 
$\BL_p$-function at $s = -1$. But we may compute, with:
$$m := 2+(p-3) p^n > 1\ \,{\rm and}\ \,(-1)_n := 1- m = -1-(p-3)p^n, $$
the following expression:
\begin{equation*}
\left \{\begin{aligned}
\BL_p(-1-(p-3)\,p^n, \chi) & = (1-p^{1+(p-3)\,p^n})\,\BL((-1)_n, \chi)) \\
& \sim \BL_p(-1, \chi),\ \,\hbox{for $n$ large enough}.
\end{aligned}\right.
\end{equation*}

\noindent
This does not give $\order \BK_2(\BZ_K)[p^\infty]$ but is related to 
$\order \BK_{2m-2}(\BZ_K)[p^\infty]$ of higher $\BK$-theory, from 
Quillen--Lichtenbaum conjecture, via the $\BL((-1)_n, \chi)$'s \cite{Ko1, Ng2}:
$$\zeta_K^{} ((-1)_n) = \zeta_K^{} (1-m) = \pm \frac{\order \BK_{2m-2}(\BZ_K)}{w_m(K)}, $$
where $w_m(K)$ is the largest integer $N$ such that $[K(\zeta_N^{}) : K] \mid m$.
One computes that for a degree-$p$ cyclic extension $K/\Q$ for $p \ne 2$, $w_m(K) \sim p$, 
which will give $v_p(\order \BK_2(\BZ_K)[p^\infty] = v_p(\zeta_K^{} (1-m))+1$.

\smallskip
Indeed, the condition $m \equiv 0 \pmod {(p-1)}$ of definition \eqref{flp} is fulfilled 
for $m = 2+(p-3) p^n$, so that relation between $p$-adic and complex 
$\BL$-functions does exist and leads, for $n$ large enough, by continuity
of $\BL_p(s, \chi)$, to
$\prod_{\chi \ne \1}\,\BL((-1)_n, \chi) \sim \prod_{\chi \ne \1}\,\BL_p(-1, \chi)$, 
since the Euler factors $1-p^{m-1} \chi(p)$ are $p$-adic units. 

\smallskip
Thus, formula of Theorem \ref{Lp} for $s = (-1)_n$ makes sense for 
computing $\BL((-1)_n, \chi)$, of constant valuation for $n \geq n_0$, 
and only depending on the computation of $\BL_p(-1, \chi)$ related to 
$\order \BK_2(\BZ_K)[p^\infty]$ and, in some sense, to $\order \CT_{K, p}$.

\smallskip
The complex analytic computation of $\BL((-1)_n, \chi)$ is done by the 
following program, with $m = 2+(p-3)p^n$, but with less immediate 
Bernoulli polynomials; so we use their definition with series;
we give also the computation of the structure of $\CT_{K, p}$:

\smallskip
\ft\begin{verbatim}
{n=2;forprime(p=5,50,m=2+(p-3)*p^n;print();forprime(ell=1,60000,
if(Mod(ell,p)!=1,next);Q=polcyclo(p);X=Mod(x,Q);g=znprimroot(ell);
P=polsubcyclo(ell,p);K=bnfinit(P,1);nu=8;Kpn=bnrinit(K,p^nu);
HKn=Kpn.cyc;TK=List;w=0;e=matsize(HKn)[2];for(k=1,e-1,
Ck=HKn[e-k+1];v=valuation(Ck,p);if(v>0,listput(TK,p^v)));
Factm=1;for(i=1,m,Factm=i*Factm);
B=(x+O(x^(m+2)))*exp(x*y/ell+O(x^(m+2)))/(exp(x+O(x^(m+2)))-1);
Bm=ell^(m-1)/m*Factm*polcoeff(B,m);L=0;dm=poldegree(Bm);
for(k=1,ell-1,a=lift(g^k);E=0;for(j=0,dm,cm=polcoeff(Bm,j);
E=E+cm*a^j);L=L+E*X^k);vm=valuation(norm(L),p);
print("p=",p," ell=",ell," T=",TK," v(K_(2m-2)Z)=",vm+1)))}
\end{verbatim}\ns

\smallskip
For instance, for $p=5$ and cyclic quintic fields of prime conductor$\ell$, 
this gives the following examples where $m=52$:

\smallskip
\ft\begin{verbatim}
ell=101 T=[5,5] v(K_(2m-2)Z)=3   ell=401  T=[5,5]     v(K_(2m-2)Z)=3
ell=151 T=[5]   v(K_(2m-2)Z)=2   ell=5351 T=[5,5,5]   v(K_(2m-2)Z)=4
ell=251 T=[5]   v(K_(2m-2)Z)=2   ell=29251T=[5,5,5,5] v(K_(2m-2)Z)=5
\end{verbatim}\ns

\smallskip
In the selected interval, $\order \CT_{K, 5}$ and $\prod_{\chi \ne \1} \BL_5(-1, \chi)$ 
have same valuation, giving $v_5(\order \BK_2(\BZ_K)) = v_5(\order \CT_{K, 5})+1$, 
whereas $\order \CT_{K, 5}$ is also obtained $p$-adically, 
replacing $\langle a \rangle \cdot \theta^{-1}\chi(a)$ by $a^{-1}\,\chi(a) = 
\langle a \rangle^{-1} \theta^{-1}\chi(a)$ in formula of Theorem \ref{Lp}. 
However, the minimal $\ell$ giving distinct $5$-valuations is given by the following data:

\smallskip
\ft\begin{verbatim}
ell=56401   T=[5,5,5,25]   v(T)=5   v(K_(2m-2)Z)=7
S0=173625530 S1=5807338145*Y S2=-7744983390*Y^2 S3=-4753111575/2*Y^3  
S4=-8281263805/2*Y^3-8281263805/2*Y^2-8281263805/2*Y-8281263805/2
S=-6517187690*Y^3-23771230585/2*Y^2+3333412485/2*Y-7934012745/2
 =5*(-1303437538*Y^3-4754246117/2*Y^2+666682497/2*Y-1586802549/2)
norm(S)=5^4*539252420555567630791570203161974917625/16=5^7*u
R=norm(S)/p=5^6*u
\end{verbatim}\ns

\smallskip
Computing $\BL_5((-1)_n, \chi)$ by means of the
formula of Theorem \ref{Lp} gives again the results:

\smallskip
\ft\begin{verbatim}
{n=2;p=5;m=2+(p-3)*p^n;Q=polcyclo(p);X=Mod(x,Q);
forprime(ell=1,5351,if(Mod(ell,p)!=1,next);fn=p*p^n*ell;
g=znprimroot(ell);c=lift(g);if(Mod(c,p)==0,c=c+ell);S=0;
for(a=1,fn/2,if(gcd(a,p*ell)!=1,next);aa=lift(Mod(a/c,fn));
la=(aa*c-a)/fn;A=lift(Mod(a,fn)^(m-1));u=znlog(Mod(a,ell),g);
S=S+(la+(1-c)/2)*A*X^u);NS=norm(S);vm=valuation(NS,p)-1;
print("p=",p," ell=",ell," v(K_(2m-2)Z)=",vm+1))}
ell=101 v(K_(2m-2)Z)=3    ell=401 v(K_(2m-2)Z)=3
ell=151 v(K_(2m-2)Z)=2    ell=601 v(K_(2m-2)Z)=2
ell=251 v(K_(2m-2)Z)=2    ell=701 v(K_(2m-2)Z)=3
\end{verbatim}\ns

\end{appendix}

\end{document}